\documentclass[11pt,reqno]{amsart}

\usepackage[matrix,arrow,curve,frame]{xy}
\usepackage{amsmath,amsthm,amssymb,enumerate,comment}
\usepackage{latexsym}
\usepackage{amscd}
\usepackage
{hyperref}
\usepackage{euscript}
\usepackage{tikz}
\usepackage{tikz-cd}
\usepackage{xy}

\newtheorem{thm}[subsection]{Theorem}
\newtheorem{defn}[subsection]{Definition}
\newtheorem{prop}[subsection]{Proposition}

\newtheorem{cor}[subsection]{Corollary}
\newtheorem{lemma}[subsection]{Lemma}
\newtheorem{conj}[subsection]{Conjecture}
\newtheorem{remark}[subsection]{Remark}
\newtheorem{assumption}[subsection]{Assumption}

\theoremstyle{definition}

\numberwithin{equation}{section}

\def\Z{{\bf Z}}

\def\C{{\bf C}}
\def\Q{{\bf Q}}

\def\cT{{\cal T}}

\def\cP{{\cal P}}
\def\cA{{\cal A}}

\def\C{{\mathbb C}}  

\def\N{{\mathbb N}}  

\def\Q{{\mathbb Q}}  
\def\R{{\mathbb R}}  

\def\Z{{\mathbb Z}}  
\def\cA{{\mathcal A}}

\def\cC{{\mathcal C}}

\def\cE{{\mathcal E}}
\def\cF{{\mathcal F}}

\def\cH{{\mathcal H}}

\def\cM{{\mathcal M}}
\def\cN{{\mathcal N}}

\def\cP{{\mathcal P}}

\def\cR{{\mathcal R}}
\def\cS{{\mathcal S}}
\def\cT{{\mathcal T}}

\def\cW{{\mathcal W}}

\def\cY{{\mathcal Y}}

\def\gg{{\mathfrak g}}

\def\gl{{\mathfrak l}}

\def\go{{\mathfrak o}}
\def\gp{{\mathfrak p}}

\def\gs{{\mathfrak s}}

\def\gA{{\mathfrak A}}
\def\gB{{\mathfrak B}}
\def\gC{{\mathfrak C}}

\def\one{\mathbf{1}}
\def\unit{{{\bf 1}_\cC}}

\DeclareMathOperator{\id}{id}

\DeclareMathOperator{\Hom}{Hom}

\DeclareMathOperator{\coev}{coev}
\DeclareMathOperator{\ev}{ev}
\DeclareMathOperator{\Id}{Id}
\DeclareMathOperator{\ptr}{ptr}
\DeclareMathOperator{\rep}{Rep}
\DeclareMathOperator{\tr}{tr}
\DeclareMathOperator{\mt}{t}
\DeclareMathOperator{\qdim}{qdim}
\DeclareMathOperator{\vir}{Vir}
\DeclareMathOperator{\com}{Com}

\newfont{\german}{eufm10}

\begin{document}

\allowdisplaybreaks
\pagestyle{plain}

\title{Simple current extensions beyond semi-simplicity}

\author{Thomas Creutzig, Shashank Kanade and Andrew R.\ Linshaw}
\thanks{E-mail: creutzig@ualberta.ca, kanade@ualberta.ca,  andrew.linshaw@du.edu}

\begin{abstract}
  Let $V$ be a simple VOA and consider a representation category of
  $V$ that is a vertex tensor category in the sense of
  Huang-Lepowsky. In particular, this category is a braided tensor
  category. Let $J$ be an object in this category that is a simple
  current of order two of either integer or half-integer conformal
  dimension.  We prove that $V\oplus J$ is either a VOA or a super
  VOA.  If the representation category of $V$ is in addition ribbon,
  then the categorical dimension of $J$ decides this parity question.
  Combining with Carnahan's work, we extend this result to
  simple currents of arbitrary order.  
Our next result is a simple sufficient criterion for lifting indecomposable objects that only depends on conformal dimensions. 
Several examples of simple current
  extensions that are $C_2$-cofinite and non-rational are then given and induced modules listed.
 \end{abstract}

\maketitle


\section{Introduction and summary of results}

Let $V$ be a simple vertex operator algebra (VOA) a fundamental and
unfortunately difficult question is whether $V$ can be extended to a larger
VOA by some of its modules.  If the representation category $\cC$
under consideration is a vertex tensor category in the sense of
Huang-Lepowsky (see \cite{HL1}), then this question is equivalent to
the existence of a haploid algebra in the category $\cC$ due to work
by Huang-Kirillov-Lepowsky \cite{HKL} following older work by
Kirillov-Ostrik \cite{KO}. This result provides a new direction of
constructing extension of VOAs.  The nicest possible extensions are
those by simple currents, that is by invertible objects in
$\cC$. Simple currents appeared first in the context of
two-dimensional conformal field theory (CFT) in the physics literature
\cite{SY, FG, GW} and have then later been introduced also for VOAs
\cite{DLM}, especially H\"ohn was able to relate the extension problem
to the categorical context \cite{Ho}.  The famous moonshine module VOA
of \cite{FLM} is a simple current extension of a certain VOA.  It can
also be constructed using vertex tensor categories \cite{H1}. Interesting
further results on simple current extensions of VOAs are for example
\cite{Y, S, Li, FRS, LaLaY, La}.  Most importantly to us, last year
Carnahan elaborated further on the simple current extension
problem \cite{C}, and he basically solved the problem for integer
weight simple currents up to extensions by self-dual ones, i.e.,
objects in the category that are their own inverses.

\subsection{Motivation}

We are interested in VOAs beyond the semi-simple setting, i.e., VOAs
whose representation category has indecomposable but reducible
modules. We also do not necessarily restrict to categories with only
finitely many simple objects. Such VOAs are sometimes called
logarithmic as they are the mathematicians' reformulation of
logarithmic CFT.  Presently, there is one well understood type of
$C_2$-cofinite but non-rational simple VOAs, the $\cW(p)$-triplet
algebras \cite{FGST, AdM, TW, CF} and the order two simple current
extension of $\cW(2)$ called symplectic fermions \cite{AA, Ab} as well
as the super triplet \cite{AdM2}.  One can now use these to construct
new VOAs via orbifolds \cite{AdLM1, AdLM2} and orbifolds of tensor
products \cite{Ab}. Also note that both $\cW(p)$ and the symplectic
fermion super VOA are rigid \cite{TW, DR}. The main objective of this
work is to develop a theory of simple current extensions of VOAs
beyond semi-simplicity. The main questions one needs to ask are
\begin{itemize}
\item Can a given VOA $V$ be extended to a larger VOA or super VOA by simple currents?
\item Which generalized modules\footnote{Generalized modules are those 
which need not admit a semi-simple action of $L(0)$ and
are graded by generalized eigenvalues of $L(0)$. } of $V$ lift to those of the extension $V_e$?
\end{itemize}
We will use our answers to these questions to construct three new
families of $C_2$-cofinite VOAs together with all modules that lift at
the end of this work.  Our main intention is however to find genuinely
new $C_2$-cofinite VOAs, that is VOAs that are not directly related to
the well-known triplet VOA.  For example consider the
following diagram

\begin{equation*}
\xymatrix{
\mathcal A_k \ar[rr]^(0.4){ \text{extension}}\ar[d]^{\text{coset}}
&& \mathcal E_k =\bigoplus\limits_{n\in \Z} J^n \ar[d]^{\text{coset}}\\
\text{Com}\left(\mathcal H, \mathcal A_k\right) \ar[rr]^(.5){\text{extension}}&&
\text{Com}\left(\mathcal H, \mathcal E_k\right).}
\end{equation*}

Here $\mathcal A_k$ is a family of VOAs with one-dimensional associated
variety \cite{A1} containing a Heisenberg sub VOA $\mathcal H$ and having a
simple current $J$ of infinite order.  Our picture is that the
Heisenberg coset $\text{Com}\left(\mathcal H, \mathcal A_k\right)$ still
has a one-dimensional associated variety while the extension
$\mathcal E_k$ only has finitely many irreducible objects. Especially
$\text{Com}\left(\mathcal H, \mathcal E_k\right)$ is our candidate for
new $C_2$-cofinite VOAs.

A natural example is $\mathcal A_k=L_k(\gs\gl_2)$ for
$k+2\in\Q_{>0}\setminus \{1, \frac{1}{2},\frac{1}{3}, \dots \}$ as it has
one-dimensional variety of modules \cite{AdM3}. In \cite{CR1, CR2} it is
conjectured that these VOAs allow for infinite order simple current
extensions $\mathcal E_k$, which then would only have finitely many simple
modules. These VOAs would be somehow unusual as they would not be of
CFT-type, and these VOAs are not our final goal, but rather
$\text{Com}\left(\mathcal H, \mathcal E_k\right)$. In the example of
$k=-1/2$ this construction would just yield $\cW(2)$ and in the case
of $k=-4/3$ it would just give $\cW(3)$ \cite{Ad1, CRW, R2}. In all
other cases we expect new $C_2$-cofinite VOAs. Another potential
candidate is the Bershadsky-Polyakov algebra whose Heisenberg coset is
studied in \cite{ACL}. In a subsequent work, we will thus develop
general properties of Heisenberg cosets beyond semi-simplicity
\cite{CLR}. This work will rely on our findings here and will then
further be used for interesting examples.

The extension problem leads to interesting number theory if restricted
to characters. Namely it seems that those types of functions that
appear in characters of logarithmic VOAs also appear in current
research of modular forms and beyond. For example, the characters of
the singlet algebra $\cM(p)$ are sometimes composed of partial theta
functions \cite{F, CM}, while their infinite order simple current
extensions, the triplets $\cW(p)$, have as characters of modules just
ordinary theta functions and their derivatives \cite{F}. Another
example is $V_k\left(\gg\gl(1|1)\right)$, while its module characters
are built out of ordinary Jacobi theta functions it has many simple
current extensions whose module characters are sometimes Mock Jacobi
forms \cite{AC, CR4}.

In the present work, we will translate results into the corresponding
statements in a braided tensor category using the theory of \cite{KO,
  HKL}. The advantage is that the categorical picture is much better
suited for proving properties of the representation category. For us
it provides a very nice way to understand the problem of the two
questions: Does a module lift to a module of the extended VOA? Is the
extension a VOA or a super VOA?  The work \cite{KO} assumes categories
to be semi-simple and focuses on algebras with trivial twist. We
believe that many of the results of \cite{KO} can be modified to our
non semi-simple setting, the most important one being the question of
rigidity in the category of the extension. We shall look at related
generalizations in future work.

\subsection{Results}

In order to describe our first result recall that braidings are the
commutativity isomorphisms which we denote by
\[
c_{A, B} : A \boxtimes B \rightarrow B\boxtimes A
\]
for objects $A, B$ in our category.  Let $J$ be a self-dual simple
current, i.e., $J\boxtimes J \cong V$. It can only give rise to a VOA
extension of $V$ if its conformal dimension $h_J$, that is the
conformal weight of its lowest weight state, is in $\frac{1}{2}\Z$. The
twist of an object in the representation category of a VOA is given by
the action of $e^{2\pi i L(0)}$ on the object, so that on a simple
module like the simple current $J$ it just acts as
$\theta_J =e^{2\pi i h_J}\Id_J$.  The balancing axiom of braided
tensor categories applied to this case reads
\[
1=\theta_V=\theta_{J\boxtimes J}= c_{J, J}\circ c_{J, J} 
\circ \left(\theta_J \boxtimes \theta_J\right)
\]
so that $\theta_J \in \{ \pm \Id_J \}$ implies
$c_{J, J} \in \{ \pm 1 \}$.  Our first result is Theorem
\ref{thm:constructingVOSA}, which is:
\begin{thm}
Assume that $V$ is a VOA satisfying the
conditions required to invoke Huang-Lepowsky-Zhang's theory. We also
assume that braiding and twist are as given by Huang-Lepowsky.
Let $J$ be a simple current such that $J\boxtimes J \cong V$.  If
$c_{J,J}=1$, $V\oplus J$ has a structure of $\frac{1}{2}\Z$-graded
vertex operator algebra and if $c_{J,J}=-1$, $V\oplus J$ has a
structure of vertex operator superalgebra.
\end{thm}
We remark that simple current extensions by self-dual simple currents
generated by a weight one primary vector were understood in the
rational, $C_2$-cofinite and CFT-type setting 20 years ago
\cite{DLM, Li}. Also the extensions of the unitary rational Virasoro
VOAs are known \cite{LaLaY}.

The proof of this theorem is very similar to the proof of Theorem 
4.1 of \cite{H6}.
For this theorem, we need a vertex tensor category of $V$ in the sense
of \cite{HL1}. The current state of the art is that the
representation category of $C_2$-cofinite VOAs with natural
additional requirements as well as subcategories of VOAs such that all
modules in this subcategory are $C_1$-cofinite  with a few
more natural additional properties are vertex tensor
categories. Especially, braiding and twist are given as developed
by Huang-Lepowsky.  For the precise requirements, see Theorems
\ref{thm:cond1} and \ref{thm:cond2} which are due to \cite{H3, HLZ,
  Miy}. For the background on the vertex tensor categories, we refer
the reader to \cite{HL1}. The construction of vertex tensor category
structure in the non semi-simple case is accomplished in \cite{HLZ} and
for a quick perspective on \cite{HLZ} and the related results, see
\cite{HL2}.

In order to understand whether the extension is a super VOA or just a
VOA, one needs to determine the braiding $c_{J, J}$. This quantity is
often not directly accessible, however there is a useful spin
statistics theorem (we adapt the name from \cite{GL} for a similar
theorem, but in the unitary conformal net setting). This spin
statistics theorem needs the notion of a trace, so it only holds in
rigid braided tensor categories, actually only in ribbon
categories. 
\begin{thm}
Assume that the tensor category of $V$ is ribbon, then,
for a self-dual simple current $J$ with conformal dimension
$h_J$,
\begin{align}
c_{J,J}\qdim(J)=e^{2\pi i h_J},
\end{align}
where $\qdim(J)=\tr_J(\Id_J)$.
\end{thm}
This is a small reformulation of Corollary
\ref{cor:spinstatistics}. 
The quantity $\qdim(J)$ is the categorical or quantum
dimension of $J$. 
In a modular tensor category, that is in the tensor
category of a regular VOA, the quantum dimension is determined from
the modular $S$-matrix coefficients as
\[
\qdim(J) =\frac{S_{J, V}}{S_{V, V}}
\]
which coincides with 
\begin{align}
\lim_{t\rightarrow 0} \frac{\text{ch}[J](\tau)}{\text{ch}[V](\tau)}
\label{eqn:Ch}
\end{align}
if the module of lowest conformal weight is $V$ itself (cf.\ \cite{DJX}),
for instance, in the case of unitary VOAs. The quantity \eqref{eqn:Ch}
is clearly non-negative and hence in this case $\qdim(J)=1$.  For
modularity of the categories of modules for vertex operator algebras
satisfying suitable finiteness and reductivity conditions, see
\cite{H3,H4}.  This means Carnahan's evenness conjecture \cite{C} is
correct for unitary regular VOAs. But beyond that there are
counterexamples: Rational $C_2$-cofinite counterexamples are our
Theorem \ref{thm:osp} as well as Theorem 10.3 of \cite{ACL}.  The
symplectic fermion super VOA is a $C_2$-cofinite but non-rational
simple current extension of the triplet VOA $\cW(2)$ graded by the
integers and hence in this case the simple current must have quantum
dimension minus one. See \cite{AA} but also \cite{CG} on symplectic
fermions. Interestingly, there are strong indications that quantum
dimensions are still determined by the modular properties of
characters in the $C_2$-cofinite setting \cite{CG} and even beyond
that \cite{CM, CMW}.

Having solved the extension problem for self-dual simple currents, we
can combine our findings with those of Carnahan \cite{C} to get
Theorem \ref{thm:generalorderJ}.
\begin{thm}
Assume that $V$ is a VOA satisfying the conditions required to invoke
Huang-Lepowsky-Zhang's theory. We also assume that braiding and twist are
as given by Huang-Lepowsky.  Let $J$ be a simple current.  Assume
that $\theta_{J^k}=\pm 1$
with $\theta_{J^{k+2}}$ having the same sign as
$\theta_{J^k}$
for all $k\in\Z$.
Then
\[
V_e= \bigoplus_{j\in G}J^j
\]
 has a natural structure of a
strongly graded vertex operator superalgebra, graded by the abelian
group $G$ generated by $J$. (For the definition of strongly graded,
we refer the reader to \cite{HLZ}.) If $G$ is finite, we get a
vertex operator superalgebra.
\end{thm}

Next, we would like to elaborate on the representation category of
simple current extensions. The works \cite{KO} and \cite{HKL} together
bring us into a good position here, since there is also a categorical
notation of VOA extension as a haploid algebra in the category as well
as many nice results on the representation category of local modules
of the haploid algebra. We only need to adapt the main theorem and its
proof of \cite{HKL} to extensions that are super VOAs, which leads us
to the notion of superalgebras in the category. Then Theorem
\ref{thm:HKLanalog} says that extensions that are super VOAs are
equivalent to super algebras in the category, and Theorem
\ref{thm:modules} tells us that the module category of the extended
VOA is equivalent to the category of local modules of the
corresponding superalgebra in the category.  This is very useful, as
we can use \cite{KO} to define a functor $\cF$ (the
induction functor) from the category to the representation category of
its superalgebra.  
For this, let $\mathcal C'\subset \mathcal C$ be the full subcategory of
generalized $V$-modules, such that objects of $\cC'$ are
subquotients of objects of the tensor ring generated by the
simple objects of $\mathcal C$.  Theorems
\ref{cor:liftingprojective} and \ref{cor:inforderJlifting} are
the following.
\begin{thm}
Let $J$ be simple current such that the extension $V_e$ exists
and let $P$ be an indecomposable generalized $V$-module, then:
\begin{enumerate}
\item If $P$ is an object of $\mathcal C'$ then $\cF(P)$ is a
generalized $V_e$-module iff $h_{J\boxtimes P}-h_J-h_P\in\Z$.
\item If $J$ is of finite order and if $P$ is an object of $\mathcal C$ such that both
$\dim(\Hom(P,P))<\infty$ and
$\dim(\Hom(J\boxtimes P, J\boxtimes P))<\infty$.  Assume also that
$L(0)$ has Jordan blocks of bounded size on both $P$ and
$J\boxtimes P$.  Then, $\cF(P)$ is a generalized $V_e$-module iff
$h_{J\boxtimes P}-h_{J}-h_{P}\in\Z$.
\end{enumerate}
\end{thm}

In other words, answering whether an indecomposable 
module $P$ of the VOA $V$ lifts to a generalized module of the simple current
extension amounts to the computation of a few conformal dimensions.

The corresonding theorem for simple modules of $C_2$-cofinite, rational, CFT-type VOAs
has been proven in \cite{Y, La}.

In practice, it is expected that $\mathcal C'$ is the category of
``most interesting $V$-modules.'' For example, in the case of $\cW(p)$ it
contains the category whose indecomposable objects consist of all simple and
all projective modules of $\cW(p)$-mod \cite{NT, TW}. Analogous
statement is true for modules of $V_k(\gg\gl(1|1))$ \cite{CR4} and we
expect this to be a generic feature of ``nice'' logarithmic VOAs.  In
the case of the Heisenberg VOA, $\mathcal C'$ is the category of
semi-simple modules; in this case the interesting infinite order
simple current extensions are lattice VOAs (of positive definite
lattices) and thus they are rational and $C_2$-cofinite. The
Heisenberg VOA has indecomposable but reducible objects, but they are
not objects in $\mathcal C'$ and they do not lift to modules of the
lattice VOA.

Carnahan titled his work \cite{C} ``Building vertex algebras from
parts,'' and indeed in the last section we construct various new
logarithmic VOAs as simple current extensions of tensor products of
known VOAs. Our main examples are three series of $C_2$-cofinite but
non-rational VOAs constructed from the tensor product of $\cW(p)$ with
a suitable second VOA. We also list interesting modules, both simple and
indecomposable, that lift to modules of the extension.  Further examples
of resulting VOAs are the small $N=4$ super Virasoro algebra at
central charge $c=-3$ as well as super VOAs associated to
$\go\gs\gp(1|2)$. A non-logarithmic example is then
$L_1(\go\gs\gp(1|2))$, which is rational (Theorem \ref{thm:osp}) and
has only two inequivalent simple modules (Corollary
\ref{cor:ospmodules}).  Finally, our results are used in proving that
the coset vertex algebras of the rational Bershadsky-Polyakov algebra
\cite{A} with its Heisenberg subalgebra are rational W-algebras of
type $A$ \cite{ACL}.

We organize this paper as follows.  In Section 2, we start with some
crucial results in braided tensor categories. Especially we derive the
``spin statistics theorem,'' as well as results that eventually allow
us to deduce the criteria for lifting modules to modules of the simple
current extension.  Section 3 is the heart of this work and contains
all the main theorems. We conclude with some examples in Section 4,
focusing on VOAs with non semi-simple representation categories.

\noindent{\bf A remark on notation}
In Section 2, when we present several general results for braided
tensor categories, we shall denote the tensor products by
$\otimes$. In Section 3, we work with vertex tensor categories, where
we use the $P(z)$-tensor products denoted by $\boxtimes_{P(z)}$ as in
\cite{HLZ}. For a fixed value of $z$, taken to be $z=1$ for
convenience, we get a braided tensor category structure and we shall
abbreviate $\boxtimes_{P(1)}$ by $\boxtimes$. In Section 4, we use
$\otimes$ yet again to denote tensor products of vertex operator
algebras (see \cite{FHL}) and their modules.  We hope that no
confusion shall arise with various ``tensor products'' used in the
paper.

\noindent {\bf Acknowledgements} 
Most of all, we all would like to thank Tomoyuki Arakawa for
discussing the simple current problem with us in the context of the
Bershadsky-Polyakov algebra \cite{ACL}.  TC is also very grateful to
Terry Gannon for constant discussions on braided tensor categories
over the last years. We are deeply indebted to Yi-Zhi Huang and James
Lepowsky for illuminating discussions regarding vertex tensor
categories.  TC is supported by an NSERC grant, the project ID is
RES0020460. SK is supported by PIMS post-doctoral fellowship.
AL is supported by the grant $\#$318755 from the Simons Foundation.

\section{Braided tensor categories}

In this section, we shall derive a spin statistics theorem for objects
in a ribbon category. We will also study a few properties of the
monodromy matrix and we will discuss the notion of a superalgebra
inside a category. Our main sources of inspiration are \cite{KO, DGNO}.

Our notation for the braiding, associativity isomorphisms, the evaluation map
and the coevaluation are
\begin{align*}
c_{A, B}&: A\otimes B \rightarrow B\otimes A, \qquad\qquad\
\cA_{X,Y,Z}: X\otimes(Y\otimes Z) \rightarrow (X\otimes Y)\otimes Z\\
\ev_X &: X^*\otimes X \rightarrow \unit,\qquad\qquad\qquad
\coev_X : \unit \rightarrow X \otimes X^*.
\end{align*}

\subsection{Spin Statistics}

Let $\cC$ be a ribbon category, that is a rigid braided tensor category with pivotal structure 
$\psi : X \rightarrow X^{**}$.
Using \cite{DGNO}, we define:
\begin{defn}
For $f\in\text{ Hom}(X\otimes Y, X\otimes Z)$ and
$g\in\text{ Hom}(Y\otimes X, Z\otimes X)$,
let
\begin{align*}
{\ptr}_X^L(f)&=
({\ev_X}\otimes\Id_Z)\circ
\cA_{X^*,X,Z}\circ
(\Id_{X^*}\otimes f)\circ
(\Id_{X^*}\otimes (\psi^{-1}_X\otimes \Id_Y))\circ
\cA^{-1}_{X^*,X^{**},Y}\circ
(\coev_{X^*}\otimes \Id_Y)
\\
{\ptr}_X^R(g)&=
(\Id_Z\otimes {\ev_{X^*}})\circ
\cA^{-1}_{Z,X^{**},X^*}\circ
((\Id_Z\otimes \psi_X)\otimes  \Id_{X^*})\circ
(g\otimes \Id_{X^*})\circ
\cA_{Y,X,X^*}\circ
(\Id_Y\otimes \coev_{X}).
\end{align*}
\end{defn}
If we take $Y=Z=\unit$ then we recover the ordinary left and
right trace of $X$. These coincide in spherical categories like a
ribbon category.

\begin{defn}
The braided structure yields a natural morphism $u_X: X\rightarrow X^{**}$ given by
\begin{align*}
X&\xrightarrow{{\Id}_X\otimes\coev_{X^*}}X\otimes(X^*\otimes X^{**})
\xrightarrow{\cA_{X,X^*,X^{**}}}(X\otimes X^*)\otimes X^{**}\\
&\xrightarrow{c_{X,X^*}\otimes\Id_{X^{**}}}(X^*\otimes X)\otimes X^{**}
\xrightarrow{\ev_X\otimes\Id_{X^{**}}}X^{**}.
\end{align*}
\end{defn}
It is actually well-known that $u_X$ is an isomorphism. 
\begin{defn}
The twist $\theta_X$ is defined by $\psi_X=u_X\theta_X$, or equivalently $\theta_X^{-1}=\psi_X^{-1}u_X$.
\end{defn}

\begin{thm}
In a ribbon category
\begin{align}
\ptr_X^L(c_{X,X}^{-1})=\theta_X^{-1}
\label{eqn:ptrLc}
\end{align}
\end{thm}
\begin{proof}
We follow the proof of Proposition 2.32 in \cite{DGNO}, except that we incorporate the 
associativity morphisms.

The left part of the following diagram commutes by naturality of
braiding and the fact that $c_{X,\unit}=\text{Id}_X$.
The right part commutes by hexagon identity.
\begin{equation*}
\xymatrix{
X = X\otimes \unit \ar[rr]^(0.5){\id_X\otimes \coev_{X^*}}\ar[d]^{\id_X=c_{X,\unit}}
&&X\otimes (X^*\otimes X^{**})\ar[d]^{c_{X,X^*\otimes X^{**}}} 
\ar[rrrr]^{c_{X,X^*}\otimes \Id_{X^{**}}\,\circ\, \cA_{X,X^*,X^{**}}}
&&
&& 
(X^*\otimes X)\otimes X^{**}
\\
X = \unit\otimes X \ar[rr]^(.5){\coev_{X^*}\otimes \id_X}&&
(X^*\otimes X^{**})\otimes X.
\ar[rrrru]_{\qquad\qquad\qquad \cA_{X^*,X,X^{**} } \, \circ\, \Id_{X^*}\otimes 
c^{-1}_{X,X^{**}}\,\circ\,  \cA^{-1}_{X^*,X^{**},X} } &&
&& 
}
\end{equation*}
Thus, we get that:
\begin{align*}
\theta_X^{-1}=\psi^{-1}_Xu_X &= 
\psi_X^{-1}\circ(\text{ev}_X\otimes \Id_{X^{**}})\circ
\cA_{X^*,X,X^{**} } \circ(\Id_{X^*}\otimes c^{-1}_{X,X^{**}})\circ 
\cA^{-1}_{X^*,X^{**},X}\circ(\coev_{X^*}\otimes \Id_X)
\end{align*}
Using naturality of associativity and braiding, we also get that
\begin{align*}
\psi_X^{-1}\circ & (\text{ev}_X\otimes \Id_{X^{**}})\circ
\cA_{X^*,X,X^{**} } \circ(\Id_{X^*}\otimes c^{-1}_{X,X^{**}}) \\
&=(\text{ev}_X\otimes \Id_{X})\circ( \Id_{X^*\otimes X}\otimes\psi^{-1}_X) \circ
\cA_{X^*,X,X^{**} } \circ(\Id_{X^*}\otimes c^{-1}_{X,X^{**}}) \\
&=(\text{ev}_X\otimes \Id_{X})\circ \cA_{X^*,X,X }\circ
( \Id_{X^*}\otimes (\Id_X \otimes\psi^{-1}_X)) \circ(\Id_{X^*}\otimes c^{-1}_{X,X^{**}}) \\
&=(\text{ev}_X\otimes \Id_{X})\circ \cA_{X^*,X,X }\circ
(\Id_{X^*}\otimes c^{-1}_{X,X})\circ ( \Id_{X^*}\otimes (\psi^{-1}_X\otimes\Id_X)) 
\end{align*}
Putting everything together,
\begin{align*}
\theta_X^{-1}
&=(\text{ev}_X\otimes \Id_{X})\circ \cA_{X^*,X,X }\circ
(\Id_{X^*}\otimes c^{-1}_{X,X})\circ ( \Id_{X^*}\otimes (\psi^{-1}_X\otimes\Id_X)) 
\circ \cA^{-1}_{X^*,X^{**},X}\circ (\coev_{X^*}\otimes \Id_X)\\
&=\ptr_X^L(c^{-1}_{X,X})
\end{align*}
\end{proof}
Taking the trace of left and right hand side of equation \eqref{eqn:ptrLc} we get
the following Corollary.
\begin{cor}
In a ribbon category
\[
\tr_{X\otimes X}\left(c_{X, X}^{-1}\right) =\tr_X\left(\theta_X^{-1}\right).
\]
\end{cor}
\begin{proof}
This follows since
$\tr_{X\otimes Y}(f) = \tr_X\left(\ptr_Y^L(f)\right)$ for any
endomorphism $f: Y\otimes X\rightarrow Y\otimes X$.
\end{proof}
\begin{remark}
Recall that our ribbon category is not necessarily semi-simple. If it
is not, then the trace might vanish on a tensor ideal. Call such an
ideal $\cP$ for projective.  If it is generated by an ambidextrous
element it allows for a modified trace on $\cP$ that we call $\mt$. In
that case one gets an analogous result
\[
\mt_{X\otimes X}\left(c_{X, X}^{-1}\right) =\mt_X\left(\theta_X^{-1}\right)
\]
for $X$ in $\cP$ as the projective trace also satisfies
$\mt_{X\otimes Y}(f) = \mt_X\left(\ptr_Y^L(f)\right)$ for any
 $f$ in End$(X\otimes Y)$ and any projective
module $X$. It does not matter wether $Y$ is projective or not.  The
importance of the modified trace in $C_2$-cofinite non-rational VOAs
is illustrated in \cite{CG}, the ideas on the modified trace there
follow \cite{GKP1, GKP2}.
\end{remark}
\begin{defn}
We call an invertible simple object a simple current and a simple current
that is its own inverse self-dual.
\end{defn}

Let $J$ be a simple current, we define the short-hand notation  $J^2:= J\otimes J$, 
and $J^{-1}$ for the inverse. 
Its categorical dimension is 
\[
\qdim(J)=\tr_J(\Id_J).
\]
It satisfies 
\[
\qdim(J)\qdim(J)=\qdim\left(J^{ 2}\right)\qquad
\text{and}\qquad \qdim(J)\qdim(J^{-1})=\qdim(\unit) =1.
\]
So the dimension of $J$ is non-zero.  Let $J$ now be self-dual and
$\theta_J \in \{\pm \id_J\}$, then it follows that also
$c_{J, J}\in \{ \pm 1 \}$ (our field is $\text{End}(\unit)$) due to the
balancing axiom of the twists.  Also the
dimension of $J$ can only be either one or minus one, namely
\begin{cor}\label{cor:spinstatistics}
Let $J$ be a self-dual simple current then 
\[
c_{J, J}\qdim(J) = \tilde\theta_J
\]
with $\theta_J=\tilde\theta_J\id_J$. 
In particular, if 
$\theta_J \in \{\pm \id_J\}$, then $c_{J, J}\in \{\pm 1\}$ and
$\qdim(J)\in\{\pm 1\}$.
\end{cor}

\subsection{Monodromy}

We now only assume that our monoidal category is braided but not necessarily ribbon. 

\begin{defn}
For objects $A,B\in\cC$, define the monodromy
$M_{A,B}:A\otimes B \rightarrow A\otimes B$ to be $c_{B,A}\circ c_{A,B}$.
\end{defn}

\begin{lemma}\label{lem:MisNatural}
Monodromy is natural. In particular, 
$$M_{(A\otimes B)\otimes C, D}\circ (\cA_{A,B,C}\otimes \Id_D)=
(\cA_{A,B,C}\otimes \Id_D)\circ M_{A\otimes(B\otimes C),D}.$$
Therefore, if $(Y\otimes J^i)_1$ and $(Y\otimes J^i)_2$ are two different
ways to parenthesize $Y\otimes J^{\otimes i}$ and $( J^i\otimes Y)_1$
and $( J^i\otimes Y)_2$ are two different ways to parenthesize
$J^{\otimes i}\otimes Y$, then $M_{(Y\otimes J^i)_1,X}=\Id$ implies
$M_{(Y\otimes J^i)_2,X}=\Id$ and $M_{X,(J^i\otimes Y)_1}=\Id$ implies
$M_{X,( J^i\otimes Y)_2}=\Id$.
\end{lemma}
\begin{proof}
Naturality of monodromy is implied by the naturality of braiding.
The rest follows.
\end{proof}

\begin{thm}\label{thm:monodromy}
The following hold for monodromy.
\begin{enumerate}
\item For objects $A,B,C \in \cC$ such that 
$M_{A,C}=\Id_{A\otimes C}$,
we have
\begin{align}
M_{A,B\otimes C} &= \cA^{-1}_{A,B,C}\circ(M_{A, B}\otimes \Id_C)\circ\cA_{A,B,C},
\label{eqn:monA-BC}\\
M_{A\otimes B, C} &= \cA_{A,B,C}\circ(\Id_A\otimes M_{B, C})\circ\cA^{-1}_{A,B,C}
\label{eqn:monAB-C}
\end{align}
\item If $M_{J,X}=\Id_{J\otimes X}$ and $M_{Y,X}=\Id_{Y\otimes X}$ 
then $M_{Y\otimes J^i,X}=\Id_{(Y\otimes J^{i})\otimes X}$,
for all positive integers $i$, regardless of how $Y\otimes J^i$ is parenthesized.
\item If $M_{X,J}=\Id_{X\otimes J}$  and $M_{X,Y}=\Id_{X\otimes Y}$ 
 then $M_{X,J^i\otimes Y}=\Id_{X\otimes(J^{i}\otimes Y)}$,
for all positive integers $i$, regardless of how $J^i\otimes Y$ is parenthesized.

\item If $M_{J,J}=\Id_{J\otimes J}$ then $M_{J^i,J^j}=\Id_{J^i\otimes J^j}$ for all $i,j \in \N$,
regardles of how $J^i$ and $J^j$ are parenthesized.

\item If $M_{J,J}=\Id_{J\otimes J}$ and $M_{J,X}=\Id_{J\otimes X}$ then
$M_{J^i,J^j\otimes X}=\Id_{J^i\otimes(J^j\otimes X)}$, for all $i,j\in\N$, regardles of how
$J^i$ and $J^j\otimes X$ are parenthesized.
\item If $J$ is an invertible object then $M_{J,X}=\Id_{J\otimes X}$ implies
$M_{J^{-1},X}=\Id_{J^{-1}\otimes X}$ and $M_{X,J}=\Id_{X\otimes J}$ implies $M_{X,J^{-1}}=\Id_{X\otimes J^{-1}}$.
\item If $J$ is an invertible object and $M_{J,J}=\Id_{J\otimes J}$ then
$M_{J^i,J^j}=\Id_{J^i\otimes J^j}$ for all $i,j\in\Z$, regardles of how $J^i$ and
$J^j$ are parenthesized.

\item If $J$ is an invertible object with $M_{J,J}=\Id_{J\otimes J}$ and $X$ is
such that $M_{J,X}=\Id_{J\otimes X}$ then $M_{J^i,J^j\otimes X}=\Id_{J^i\otimes(J^j\otimes X)}$, for all
$i,j\in\Z$, regardles of how $J^i$ and $J^j\otimes X$ are parenthesized.
\end{enumerate}
\end{thm}
\begin{proof}

For \eqref{eqn:monA-BC}, we proceed as follows.
By the hexagon diagram, we get:
\begin{align*}
M_{A,B\otimes C}&=c_{B\otimes C,A}\circ c_{A,B\otimes C} \\
&=   (\cA^{-1}_{A,B,C}\circ (c_{B,A}\otimes \Id_C) \circ 
\cA_{B,A,C}\circ(\Id_B\otimes c_{C,A})\circ\cA^{-1}_{B,C,A})\\
&\quad \circ
(\cA_{B,C,A} \circ (\Id_B\otimes c_{A,C})\circ\cA^{-1}_{B,A,C}
\circ(c_{A,B}\otimes\Id_C)\circ\cA_{A,B,C})\\
&=\cA^{-1}_{A,B,C}\circ (c_{B,A}\otimes \Id_C)
\circ(c_{A,B}\otimes\Id_C)\circ\cA_{A,B,C}.
\end{align*}
For \eqref{eqn:monAB-C} we proceed similarly.
Again by the hexagon diagram, we get:
\begin{align*}
M_{A\otimes B, C}&=c_{C, A\otimes B}\circ c_{A\otimes B, C} \\
&= (\cA_{A,B,C} \circ (\Id_A\otimes c_{C,B})
\circ\cA^{-1}_{A,C,B}\circ(c_{C,A}\otimes\Id_B)\circ\cA_{C,A,B})\\
&\quad \circ
  (\cA^{-1}_{C,A,B}\circ (c_{A,C}\otimes \Id_B) 
\circ \cA_{A,C,B}\circ(\Id_A\otimes c_{B,C})\circ\cA^{-1}_{A,B,C})\\
&=\cA_{A,B,C}\circ (\Id_A\otimes c_{C,B})
\circ(\Id_A\otimes c_{B,C})\circ\cA^{-1}_{A,B,C}.
\end{align*}

We first prove (2) when
$Y\otimes J^i= (\cdots ((Y\otimes J)\otimes J )\cdots \otimes J)$ for
all positive integers $i$.  Note that \eqref{eqn:monAB-C} implies
$M_{Y \otimes J,X} = \Id_{(Y\otimes J)\otimes X}$.  Therefore, by
induction on $i$, we conclude that
$M_{Y\otimes J^i,X}=\Id_{(Y\otimes J^i)\otimes X}$ for all
$i\in\Z_{+}$. Using Lemma \ref{lem:MisNatural} we can now get the
result for all different ways of parenthesizing $Y\otimes J^i$.

Using \eqref{eqn:monA-BC}, (3) follows in complete analogy.

For (4), note that the assertion holds if $i=0$ or $j=0$ because
$c_{{\unit},X}=\Id = c_{X,{\unit}}$ and hence
$M_{{\unit},X}=M_{X,{\unit}}$ for any object $X$.  Now let
$i,j\geq 1$.  If $i=j=1$, then the assertion follows by assumption
that $M_{J,J}=\Id$.  If $i=1$ or $j=1$, the claim follows by using
$Y={\unit}$ and $X=J$ in (2) and (3).  With this we have proved (4)
for $i\in\{0,1\}$ or $j\in\{0,1\}$.  Now if $i,j\geq 2$, using (2)
with $Y=J, X=J^j$, we obtain that $M_{J\otimes J^{i-1},J^j}=\Id$
regardless of how $J\otimes J^{i-1}$ and $J^j$ are parenthesized.

For (5), note that the assertion holds if $i=0$.  Indeed,
$c_{{\unit},X}=c_{X,{\unit}}=\Id_X$ results in $M_{{\unit},X}=\Id$;
which combined with (3) yields $M_{{\unit},J^j\otimes X}=\Id$.  The
assertion also holds if $j=0$ by taking $Y={\unit}$ in (2).  Using
this and (3), we obtain that
$M_{J^i,(J^j)\otimes X}=\cA^{-1}\circ(M_{J^i,J^j}\otimes \Id)\circ
\cA$
which in turn equals $\Id$ becuase of (4).  Now use Lemma
\ref{lem:MisNatural} to get the result for all parenthesizings of
$J^j\otimes X$.

For (6):
\begin{align*}
\Id_{X} &= M_{{\unit},X} \\
&= \cA_{J,J^{-1},X}\circ M_{J\otimes J^{-1},X}\circ \cA^{-1}_{J,J^{-1},A}\\
&= \cA_{J,J^{-1},X}\circ (\Id\otimes M_{J^{-1},X})\circ\cA^{-1}_{J,J^{-1},A}.
\end{align*}
Hence, $M_{J^{-1},X}$ must be $\Id$.
We proceed similarly for the rest.

For (7), using (6) we get that $M_{J,J}=\Id$ implies
$M_{J^{-1},J}=\Id$ and $M_{J,J^{-1}}=\Id$, either of which leads to
$M_{J^{-1},J^{-1}}=\Id$. The rest can be easily obtained as in the
proof of (2), (3) and (4).

Lastly, (8) is obtained by following the steps in the proof of (5),
(6) and (7).
\end{proof}

We will use the proposition above to give a lifting criterion for
simple modules in Corollary \ref{cor:liftingsimple}.  Next, we provide
some useful lemmata which will help us strengthen Corollary
\ref{cor:liftingsimple} to some indecomposable modules.

\begin{lemma}\label{lem:dimhomsame}
Let $J$ be a simple current. Then, for any $P$ and $X$,
such that $\dim(\Hom(P,X))< \infty$ and either 
$\dim(\Hom(P\otimes J, X\otimes J))< \infty$ or 
$\dim(\Hom(J\otimes P, J\otimes X))< \infty$, we have that
\begin{align}\label{eqn:dimhom}
\dim(\Hom(P,X))=\dim(\Hom(P\otimes J, X\otimes J))=
\dim(\Hom(J\otimes P, J\otimes X)).
\end{align}
\end{lemma}
\begin{proof}
Since braiding is an isomorphism, 
$\dim(\Hom(P\otimes J, X\otimes J))< \infty$ 
if and only if
$\dim(\Hom(J\otimes P, J\otimes X))< \infty$.

The conclusion holds if $J={\unit}$. That is,
$\bullet\otimes\Id_{{\unit}}$ is an isomorphism.  Fix an
isomorphism $g:{{\unit}}\rightarrow J\otimes J^{-1}$.  We first prove that
$\dim(\Hom(P,X))\leq \dim(\Hom(P\otimes J, X\otimes J))$.
Consider the map
$\bullet\otimes\Id_J: \Hom(P,X)\rightarrow \Hom(P\otimes J,
X\otimes J)$.  If $f\otimes\Id_J=0$, then,
\begin{align*}
0 &= \cA_{X,J,J^{-1}}^{-1}\circ((f\otimes\Id_J)\otimes\Id_{J^{-1}})\\
&= f\otimes(\Id_{J}\otimes\Id_{J^{-1}})=f\otimes(\Id_{J\otimes J^{-1}})\\
&=(\Id_X\otimes g)\circ (f\otimes\Id_{\unit})\circ(\Id_P\otimes g^{-1}).
\end{align*}
Therefore, $f\otimes\Id_{\unit}=0$. Since
$\bullet\otimes\Id_{\unit}$ is an isomorphism, $f=0$.

We turn to the converse direction. Replacing $J$ by $J^{-1}$ in the argument above,
we have that
\begin{align*}
\dim(\Hom(P\otimes J,X\otimes J))\leq 
\dim(\Hom((P\otimes J)\otimes J^{-1},(X\otimes J)\otimes J^{-1})).
\end{align*}
However, since associativity is an isomorphism and since
$J\otimes J^{-1}\cong {{\bf 1}_\cC}$
\begin{align*}
\Hom((P\otimes J)\otimes J^{-1},(X\otimes J)\otimes J^{-1})
&\cong\Hom(P\otimes (J\otimes J^{-1}),X\otimes (J\otimes J^{-1}))\\
&\cong\Hom(P\otimes {{\bf 1}_\cC},X\otimes {{\bf 1}_\cC})\\ 
&\cong\Hom(P,X).
\end{align*}
For showing $\dim(\Hom(P,X))=\dim(\Hom(J\otimes P, J\otimes X))$, 
one proceeds similarly.
\end{proof}

\begin{lemma}\label{lem:JPmon}
Let $J$ be a finite order simple current such that $J^N\cong{\unit}$ 
for some $N\in\Z_+$. Let $P$ be any
object such that $\dim(\Hom(P,P))<\infty$
and 
$\dim(\Hom(J\otimes P, J\otimes P))<\infty$.  Assume
that $M_{J,P}=\lambda\Id_{J\otimes P}+\pi$ where $\pi$ is a nilpotent
endomorphism of $J\otimes P$.  Then,
$\pi=0$, equivalently,
$M_{J,P}$ is a semi-simple endomorphism. Moreover,
$\lambda^N=1$.
\end{lemma}
\begin{proof}
Lemma \ref{lem:dimhomsame} in fact shows that $\Id_J\otimes\bullet$
provides an isomorphism $\Hom(P,P)\cong\Hom(J\otimes P, J\otimes P)$
and we conclude that $\pi = \Id_J\otimes\nu$ for some nilpotent
endomorphism $\nu$ of $P$.

We claim that for any $n\in\Z_+$, regardless of how
$J^n$ is parenthesized,
\begin{align}\label{eqn:Mbinom}
M_{J^n,P}=\sum\limits_{i=0}^n {n\choose i} \lambda^i \Id_{J^n}\otimes \nu^{n-i}.
\end{align}
Equation \eqref{eqn:Mbinom} holds for $n=0$ since $M_{{\bf 1}_\cC,X}=\Id_X$
for any object $X$ and for $n=1$ by assumption.
We proceed by induction. Assume that the claim holds for some
$n\in\Z_+$.
By Lemma \ref{lem:MisNatural}, it is enough to prove the claim when $J^n$
is parenthesized so that $J^n=J\otimes J^{n-1}$.
Exactly as in the proof of Theorem \ref{thm:monodromy},
\begin{align*}
M_{J\otimes J^n, P}&=c_{P, J\otimes J^n}\circ c_{J\otimes J^n, P} \\
&= \cA_{J,J^n,P} \circ (\Id_J\otimes c_{P,J^n})
\circ\cA^{-1}_{J,P,J^n}\circ (M_{J,P}\otimes \Id _{J^n})
\circ \cA_{J,P,J^n}\circ(\Id_J\otimes c_{J^n,P})\circ\cA^{-1}_{J,J^n,P}\\
&=\cA_{J,J^n,P} \circ (\Id_J\otimes c_{P,J^n})
\circ\cA^{-1}_{J,P,J^n}\circ (\lambda\Id_{J\otimes P}\otimes \Id _{J^n})
\circ \cA_{J,P,J^n}\circ(\Id_J\otimes c_{J^n,P})\circ\cA^{-1}_{J,J^n,P}
\\
&\quad + \cA_{J,J^n,P} \circ (\Id_J\otimes c_{P,J^n})
\circ\cA^{-1}_{J,P,J^n}\circ ((\Id_J\otimes \nu)\otimes \Id _{J^n})
\circ \cA_{J,P,J^n}\circ(\Id_J\otimes c_{J^n,P})\circ\cA^{-1}_{J,J^n,P}.
\end{align*}
However, using naturality of braiding and associativity and
using the induction hypothesis, we observe that:
\begin{align*}
\cA_{J,J^n,P} \circ & (\Id_J\otimes c_{P,J^n})
\circ\cA^{-1}_{J,P,J^n}\circ (\lambda\Id_{J\otimes P}\otimes \Id _{J^n})
\circ \cA_{J,P,J^n}\circ(\Id_J\otimes c_{J^n,P})\circ\cA^{-1}_{J,J^n,P} \\
&=\lambda \cA_{J,J^n,P} \circ  (\Id_J\otimes M_{J^n,P})
\circ\cA^{-1}_{J,J^n,P}\\
&=\lambda \cA_{J,J^n,P} \circ  \left(\Id_J\otimes\left(
\sum\limits_{i=0}^n {n\choose i} \lambda^i \Id_{J^n}\otimes \nu^{n-i} \right)
\right)\circ\cA^{-1}_{J,J^n,P}\\
&= \Id_J\otimes\left(
\sum\limits_{i=0}^n {n\choose i} \lambda^{i+1} \Id_{J^n}\otimes \nu^{n-i} \right)\\
&= \sum\limits_{i=0}^n {n\choose i} \lambda^{i+1} \Id_{J\otimes J^n}\otimes \nu^{n-i}.
\end{align*}
and 
\begin{align*}
\cA_{J,J^n,P} \circ & (\Id_J\otimes c_{P,J^n})
\circ\cA^{-1}_{J,P,J^n}\circ ((\Id_J\otimes \nu)\otimes \Id _{J^n})
\circ \cA_{J,P,J^n}\circ(\Id_J\otimes c_{J^n,P})\circ\cA^{-1}_{J,J^n,P}\\
&=\cA_{J,J^n,P} \circ  (\Id_J\otimes c_{P,J^n})
\circ (\Id_J\otimes (\nu\otimes \Id _{J^n}))
\circ(\Id_J\otimes c_{J^n,P})\circ\cA^{-1}_{J,J^n,P}\\
&=\cA_{J,J^n,P} \circ  (\Id_J\otimes (\Id_{J^n}\otimes \nu))\circ  (\Id_J\otimes M_{J^n,P})
\circ\cA^{-1}_{J,J^n,P}\\
&=\cA_{J,J^n,P} \circ  (\Id_J\otimes (\Id_{J^n}\otimes \nu))\circ  
\left(\Id_J\otimes\left(
\sum\limits_{i=0}^n {n\choose i} \lambda^i \Id_{J^n}\otimes \nu^{n-i} \right)
\right)\circ\cA^{-1}_{J,J^n,P}\\
&=\cA_{J,J^n,P} \circ  
\left(\Id_J\otimes\left(
\sum\limits_{i=0}^n {n\choose i} \lambda^i \Id_{J^n}\otimes \nu^{n-i+1} \right)
\right)\circ\cA^{-1}_{J,J^n,P}\\
&=
\sum\limits_{i=0}^n {n\choose i} \lambda^i \Id_{J\otimes J^n}\otimes \nu^{n+1-i}. 
\end{align*}
Combining the two, we immediately get equation \eqref{eqn:Mbinom} for $n+1$.

Now, using equation \eqref{eqn:Mbinom} for $n=N$,
we get:
\begin{align*}
\Id_P = M_{{\bf 1}_\cC,P}=M_{J^N,P}=\sum\limits_{i=0}^N {N\choose i}
\lambda^i \Id_{J^N}\otimes \nu^{N-i}. 
\end{align*}
However, since $\nu$ is nilpotent, we immediately conclude that 
$\nu=0$ and $\lambda^N=1$.
\end{proof}

\subsection{$\cC$-superalgebras}

In this section, we generalize the notion of $\cC$-algebra of Kirillov-Ostrik to $\cC$-superalgebra. 
We closely follow their notation and results \cite{KO}.

For the rest of the work, we assume the category $\cC$ to be abelian,
and we assume that $\otimes$ naturally distributes over $\oplus$.
This distributivity will hold for the categories we shall consider,
thanks to Proposition 4.24 of \cite{HLZ}.

\begin{defn}
A $\cC$-superalgebra is an object $A = A^0\oplus A^1\in\cC$ ($A^0,A^1$ are objects in $\cC$)
with morphisms
\begin{align}
\mu: A \otimes A &\rightarrow A \\
\iota: \unit &\hookrightarrow{} A^0.
\end{align}
Such that the following conditions hold.  
\begin{enumerate}
\item $\mu$ respects the $\frac{1}{2}\Z$-grading: 
$\mu(\theta \otimes \theta)=\theta\circ\mu$.
\item $\mu$ respects the $\Z_2$-grading: 
$\mu(A^i\otimes A^j)\rightarrow A^{i+j}\hookrightarrow A$.
\item Associativity:
$$\mu \circ (\mu\otimes \Id_A)\circ \cA = \mu\circ(\Id_A\otimes\mu)$$
\item 
Commutativity:
$$\mu\vert_{A^i\otimes A^j} = (-1)^{ij}\cdot \mu \circ c_{A^i, A^j}$$
\item Unit:
$$ \mu\circ (\iota_A\otimes \Id_A)\circ \ell_A^{-1}=\Id_A$$
where
$$\ell_A: \unit\otimes A \rightarrow A$$
is the left unit isomorphism.
\end{enumerate}
Such an algebra is called \emph{haploid} if it has
\begin{enumerate}
\item[(6)] Uniqueness of unit: 
$$\text{dim\,Hom}(\unit,A)=1. $$
\end{enumerate}
\end{defn}

Following \cite{KO}, we define a natural category for representations
of a $\cC$-superalgebra.
\begin{defn}
Let $(A=A^0 \oplus A^1,\mu,\iota)$ be a $\cC$-superalgebra.  Define
a category $\rep A$ as follows.  The objects are pairs
$(W = W^0\oplus W^1,\mu_W)$, where $W^0, W^1\in\cC$,
$$\mu_W: A\otimes W \cong \bigoplus_{i,j\in\Z/2\Z}A^i\otimes W^j \rightarrow W$$
is a morphism satisfying:
\begin{enumerate}
\item $\mu_W: A^i\otimes W^j \rightarrow W^{i+j\pmod{2}},$
\item $\mu_W\circ(\mu \otimes \Id_W)\circ\cA = 
\mu_W\circ(\Id_A\otimes \mu_W): A\otimes(A\otimes W)\rightarrow W,$
\item $\mu_W\circ(\iota_A\otimes\Id_W) = \ell_W : \unit\otimes W \rightarrow W$.
\end{enumerate}
The morphisms are defined as:
\begin{align}
\Hom_{\rep A}&((M,\mu_M),(N,\mu_N)) \nonumber \\
&=\{ \varphi\in\Hom_\cC(M,N)\,|\,\mu_N\circ(\Id_A\otimes\varphi) 
=\varphi\circ\mu_M: A\otimes M\rightarrow N \}
\nonumber
\end{align}
\end{defn}

\begin{defn}
Define $\rep^0 A$ to be the full subcategory of $\rep A$ consisting
of objects $(W,\mu_W)$ such that
$$\mu_W\circ(c_{W,A}\circ c_{A,W}) = \mu_W:A\otimes W\rightarrow W.$$
\end{defn}

\begin{defn}
Given a $\cC$-superalgebra $(A=A^0\oplus A^1,\mu,\iota)$, define
$$\cF(X) = (A\otimes X = A^0\otimes X \oplus A^1\otimes X, 
(\mu\otimes\Id_X)\circ\cA_{A,A,X} ),$$
$$\cF(f)= \Id_A\otimes f $$
for $X$ an object in $\cC$ and $f$ a morphism.
\end{defn}

\begin{thm}
$\cF$ is a functor from  $\cC$ to $\rep A$.
\end{thm}
\begin{proof}
Let $W$ be an object in $\cC$. We now prove that $\cF(W)$ is an object of $\rep A$.
Since $\mu:A^i\otimes A^j \rightarrow A^{i+j\pmod{2}}$, it is clear
that
$\mu_W:A^i\otimes (A^j\otimes W) \rightarrow A^{i+j\pmod{2}}\otimes  W$.
Therefore, condition (1) is satisfied. 
For condition (2), consider
the following commuting diagram,
where the unlabeled arrows correspond to associativity isomorphisms,
obtained by using the  pentagon diagram,
naturality of associativity and by the associativity of $\mu$.
\begin{align*}
\xymatrixcolsep{4pc}
\xymatrix{
&   A\otimes(A\otimes(A\otimes W)) \ar[dr] \ar[d]& & \\
A\otimes(A\otimes W)\ar[d] &
(A\otimes A)\otimes(A\otimes W)\ar[l]_{\mu\otimes( \Id\otimes\Id)} \ar[d]  &  
A\otimes((A\otimes A)\otimes W)\ar[d] 
\ar[dr]^{\Id\otimes(\mu\otimes \Id)} & \\
(A\otimes A)\otimes W \ar[d]_{\mu\otimes\Id} &
((A\otimes A)\otimes A)\otimes W \ar[l]^{(\mu\otimes\Id)\otimes \Id} &  
(A\otimes (A\otimes A))\otimes W \ar[l] 
\ar[d]_{(\Id\otimes\mu)\otimes\Id} & A\otimes(A\otimes W)\ar[dl]\\
A\otimes W & & (A\otimes A)\otimes W \ar[ll]_{\mu\otimes \Id} & 
}
\end{align*}
This commutative diagram immediately establishes (2).

For (3), we have:
\begin{align*}
\ell_{A\otimes W}&=(\ell_A\otimes \Id)\circ\cA_{{\bf 1}_\cC,A,W}\\
&=((\mu\circ (\iota_A\otimes\Id))\otimes \Id)\circ\cA_{{\bf 1}_\cC,A,W}\\
&=(\mu\otimes\Id)\circ((\iota_A\otimes\Id)\otimes \Id)\circ\cA_{{\bf 1}_\cC,A,W}\\
&=(\mu\otimes\Id)\circ\cA_{A,A,W}\circ(\iota_A \otimes(\Id \otimes \Id)),
\end{align*}
where the first equality follows by the properties of left unit, the
second property follows by the left unit property of $A$ and the last
equality follows by naturality of associativity.

Now let $f: U\rightarrow W$ be a morphism in $\cC$.
Let $\varphi = \Id_A\otimes f: \cF(U)=A\otimes U\rightarrow A\otimes W=\cF(W)$.
Then,
\begin{align*}
\mu_{\cF(W)}\circ(\Id_A\otimes\varphi) 
&= (\mu\otimes \Id_W) \circ \cA_{A,A,W}\circ (\Id_A\otimes(\Id_A\otimes f)) \\
&= (\mu\otimes \Id_W) \circ ((\Id_A\otimes\Id_A)\otimes f)\circ \cA_{A,A,W} \\
&=  (\Id_A\otimes f)\circ(\mu\otimes \Id_W) \circ \cA_{A,A,W} \\
&=\varphi\circ \mu_{\cF(U)},
\end{align*}
where we have used naturality of associativity in the second equality.
\end{proof}

\section{Simple current extensions and algebras}

\begin{defn}
A vertex operator superalgebra is a triple $(V, \mathbf{1},\omega,Y)$, where 
$V$ has compatible gradings by $\frac{1}{2}\Z$ and $\Z_2$,
i.e., 
\begin{align}
V = V^0 \oplus V^1 = \bigoplus_{n\in\frac{1}{2}\Z}V^0_n \oplus \bigoplus_{n\in\frac{1}{2}\Z}V^1_n,
\end{align}
$Y$ is a map
\begin{align}
Y : V\otimes V \rightarrow V[[x,x^{-1}]],
\end{align}
such that the following axioms are satisfied. We let $Y(\omega,x)=\sum_{n\in\Z}L(n)x^{-n-2}$.

\begin{enumerate}
\item Axioms for grading: 

Lower truncation: $V_n=0$ for all sufficiently negative $n$.

Finite dimensionality: Each $V_n$ is finite dimensional.

$L(0)$-grading property: $V_n=\{v\in V\,|\,L(0)v=nv\}$.

$\mathbf{1}\in V^0_0$, $\omega\in V^0_2$.

\item Axioms for vacuum:

Left-identity property: $Y({\bf 1}, x)v = v$ for all $v\in V$.

Creation property: $\lim_{x\rightarrow 1}Y(v,1)\mathbf{1}$ exists and equals $v$ for all $v\in V$.

\item $L(-1)$-derivative property: $[L(-1), Y(v,x)]=Y(L(-1)v,x)=\dfrac{d}{dx}Y(v,x)$.

\item Virasoro relations: $[L(m),L(n)]=(m-n)L(m+n)+\dfrac{m^3-m}{12}\delta_{m,-n}c$

\item Jacobi identity: For $u \in V^i$, $v\in V^j$,
\begin{align*}
x_0^{-1}\delta&\left(\dfrac{x_1-x_2}{x_0}\right)Y(u,x_1)Y(v,x_2)
-(-1)^{ij}x_0^{-1}\delta\left(\dfrac{-x_2+x_1}{x_0}\right)Y(v,x_2)Y(v,x_1)\\
&=x_2^{-1}\delta\left(\dfrac{x_1-x_0}{x_2}\right)Y(Y(u,x_0)v,x_2).
\end{align*}
\end{enumerate}

\end{defn}

\begin{defn} 
A $\frac{1}{2}\Z$-graded vertex operator algebra is a vertex
operator superalgebra $V$ such that $V^1=0$.
\end{defn}

\begin{defn}
A vertex operator algebra is a $\frac{1}{2}\Z$-graded vertex
operator algebra that is in fact $\Z$-graded.
\end{defn}

\begin{defn}
Consider a vertex operator superalgebra $(V, \mathbf{1},\omega,Y)$.
A $V$-module is a vector space $W$ with compatible gradings by 
$\R$ and $\Z/2\Z$ i.e., 
\begin{align}
W = W^0 \oplus W^1 = \bigoplus_{n\in\R}W^0_n \oplus \bigoplus_{n\in\R}W^1_n,
\end{align}
equipped with a vertex operator map $Y_W$,
\begin{align}
Y_W : V\otimes W \rightarrow W[[x,x^{-1}]],
\end{align}
such that the following axioms are satisfied. We denote the modes of
$Y_W(\omega,x)$ by $L(n)$.

\begin{enumerate}
\item Axioms for grading: 

Lower truncation: $W_n=0$ for all  sufficiently negative $n$.

Finite dimensionality: Each $W_n$ is finite dimensional.

$L(0)$-grading property: $W_n=\{w\in W\,|\,L(0)w=nw\}$.

$\Z/2\Z$-grading compatibility: $Y_W: V^i \otimes W^j \rightarrow W^{i+j\pmod{2}}[[x,x^{-1}]]$.

\item Axioms for vacuum:

Left-identity property: $Y_W({\bf 1}, x)w = w$ for all $w\in W$.

\item $L(-1)$-derivative property: $[L(-1), Y(v,x)]=Y(L(-1)v,x)=\dfrac{d}{dx}Y(v,x)$.

\item Jacobi identity: For $u\in V^i$, $v \in V^j$,
\begin{align*}
x_0^{-1}\delta&\left(\dfrac{x_1-x_2}{x_0}\right)Y_W(u,x_1)Y_W(v,x_2)
-(-1)^{ij}x_0^{-1}\delta\left(\dfrac{-x_2+x_1}{x_0}\right)Y_W(v,x_2)Y_W(v,x_1)\\
&=x_2^{-1}\delta\left(\dfrac{x_1-x_0}{x_2}\right)Y_W(Y(u,x_0)v,x_2).
\end{align*}
\end{enumerate}
\end{defn}

\begin{remark}
For definitions involving complex variables instead of the formal
variables, refer to \cite{H6}.
\end{remark}

\begin{defn}
A module $W$ is called a generalized $V$-module if the $\R$-grading
on $W$ is by generalized eigenvalues of $L(0)$, i.e., $W$ is a
direct sum of generalized eigenspaces of $L(0)$.  Thus, a generalized
module $W$ is in fact a grading-restricted generalized module in the
sense of \cite{H6}.
\end{defn}

\begin{assumption}
We will work with the following assumption in the next few sections.
Assume that $V$ is a vertex operator algebra satisfying the
conditions required to invoke Huang-Lepowsky-Zhang's theory. We also
assume that braiding and twist are as given by Huang-Lepowsky-Zhang.
We assume that the tensor bifunctor is chosen to be $\boxtimes_{P(1)}$,
which we abbreviate to be $\boxtimes$.
\end{assumption}

\begin{defn}
Recall Definition 3.10, 4.2 and 4.13 of logarithmic intertwining
operators, $P(z)$-intertwining maps and $P(z)$-tensor products,
respectively, from \cite{HLZ}
and the definitions of rationality of products, rationality of
iterates, commutativity and associativity for the vertex operator
map and the module map from \cite{H6}.
\end{defn}

\begin{thm}\label{thm:constructingVOSA}
Let $J$ be a simple current such that 
$J\boxtimes J \cong V$ (which implies that $V$ is simple, see \cite{CLR}) and
$\theta_J=\pm \Id_J$ (which implies that $c_{J,J}\in\{\Id,-\Id\}$ by
balancing).  If $c_{J,J}=1$, $V\oplus J$ has a
structure of $\frac{1}{2}\Z$-graded vertex operator algebra and if
$c_{J,J}=-1$, $V\oplus J$ has a structure of vertex operator
superalgebra.
\end{thm}

\begin{proof}
Structurally, the proof is similar to the proof of Theorem 4.1 from \cite{H6}.
Note the following implications of the assumptions: Since $J$ is simple,
$L(0)$ acts semisimply on $J$, moreover,
since $\theta_J=\pm\Id_J$, $J$ is graded either by $\Z$ or by
$\frac{1}{2}+\Z$.  Since $L(0)$ acts semi-simply on both $V$ and
$J$, any logarithmic intertwining operator of the type $ C\choose{A\,B}$ where
$A,B,C\in\{V,J\}$ is free of logarithms, cf.\ Remark 3.23 of
\cite{HLZ}.  Since $J$ is simple, we have the following fusion
rules:
\begin{align}
\cN_{J,V}^J=\cN_{V,J}^J=\dim(\Hom(V\boxtimes J, J))=\dim(\Hom(J, J))=1.
\end{align}
Also, by assumption, 
\begin{align}
\cN_{J,J}^{J\boxtimes J}= \dim(\Hom(J\boxtimes J, J\boxtimes J))
=\dim(\Hom(J\boxtimes J, V))
=\cN_{J,J}^V=
\dim(\Hom(V,V))=1.
\end{align}

Fix an isomorphism $j:J\boxtimes J \rightarrow V$.

Let $\cY_{\boxtimes}$ be the (non-zero) intertwining operator
corresponding to the intertwining map $\boxtimes$ of type
$J\boxtimes J \choose{J\, J}$.  Let
\begin{align*}
\cY = j\circ \cY_\boxtimes.
\end{align*}
It is clear that $\cY$ is the non-zero intertwining operator of type
${V \choose { J\, J}}$ corresponding to the $P(1)$-intertwining map
$\overline{j}\circ \boxtimes: J\otimes J \rightarrow \overline{V}$.
Since $J$ is graded either by $\Z$ or by $1/2 +\Z$, $\cY$ has
only integral powers of the formal variable.

Consider the intertwining operator $Y_e$ of type $V\oplus J \choose {V\oplus J\,\,V\oplus J}$
defined by
\begin{align*}
Y_e(v_1\oplus j_1,x)&(v_2\oplus j_2)  
= (Y(v_1,x)v_2 + \cY(j_1,x)j_2)\oplus (Y_J(v_1,x)j_2 + e^{xL(-1)}Y_J(v_2,-x)j_1)
\end{align*}
Note that $Y_e$ also has only integral powers of $x$.

First, we analyze the braiding in order to relate $c_{J,J}$ with the
skew-symmetry.  We will need this information to prove the
associativity for $Y_e$.  The braiding is characterized by (cf.\
equation (3.9) of \cite{HKL}):
\begin{align*}
\overline{\cR}_{J\boxtimes J}(j_1\boxtimes j_2)=
e^{L(-1)}\overline{\cT}_{\gamma_1^-}(j_2\boxtimes_{P(-1)}j_1),
\end{align*}
where $\gamma_1^-$ is a path in $\overline{\mathbb{H}}\backslash\{0\}$
from $-1$ to $1$, and correspondingly, $\cT_{\gamma_1^-}$ is the
parallel transport isomorphism from $J\boxtimes_{P(-1)}J$ to
$J\boxtimes J$.  Recall that $\cY_{\boxtimes}$ is the intertwining
operator corresponding to the intertwining map $\boxtimes$ of
type $J\boxtimes J \choose{J\, J}$.  Then, (cf.\ equation (3.11) of
\cite{HKL}),
\begin{align*}
\overline{\cR}_{J\boxtimes J}(j_1\boxtimes j_2) = 
e^{L(-1)}\cY_\boxtimes(j_2,e^{i\pi})j_1.
\end{align*}
Due to our assumption, we know that
$\cR_{J\boxtimes J} = c_{J,J}\text{Id}_{J\boxtimes J}$ on $J\boxtimes J$ and hence,
\begin{align}
c_{J,J}(j_1\boxtimes j_2) = c_{J,J}\cY_\boxtimes(j_1,1)j_2=
e^{L(-1)}\cY_\boxtimes(j_2,e^{i\pi})j_1=e^{L(-1)}\cY_\boxtimes(j_2,-1)j_1.
\end{align}
Composing with $\bar{j}$, (recall that we have fixed an isomorphism
$j:J\boxtimes J \rightarrow V$), we get:
\begin{align}
c_{J,J}\cY(j_1,1)j_2=e^{L(-1)}\cY(j_2,-1)j_1\label{eqn:cJJskew}
\end{align}
Note that this can also be written as
\begin{align}
c_{J,J}Y_e(j_1,1)j_2=e^{L(-1)}Y_e(j_2,-1)j_1.\label{eqn:skewforY1}
\end{align}

Now we move to the associativity of $Y_e$.  We would like to prove
that for all $u,v,w\in X$ and $|z_1|>|z_2|>|z_1-z_2|>0$,
\begin{align}
Y_e(u,z_1)Y_e(v,z_2)w = Y_e(Y_e(u,z_1-z_2)v,z_2)w\label{eqn:assoc}.
\end{align}

There are a few cases: If all $u,v,w$ are in $V$ then all
of $Y_e$ are equal to the vertex operator map $Y$.  So the equality of
right-hand sides follows from the associativity for $Y$.  Similarly,
if exactly one of $u,v,w$ is in $J$ then equality of the right-hand
sides follows from the properties of the module map $Y_J$.  If exactly
two of $u,v,w$ are in $J$ then equality follows from the properties of
the intertwining operator $\cY$, see \cite{FHL}.

The tricky part is when all $u,v,w$ are in $J$. In this case, we would like 
to prove that:
\begin{align*}
Y_e(j_1,z_1) \cY(j_2,z_2)j_3 = Y_e(\cY(j_1,z_1-z_2)j_2,z_2)j_3
\end{align*}
Using Theorem 9.24 of \cite{HLZ}, we know this statement up to a
constant: We know that there exist intertwining operators $\cY^1$ and
$\cY^2$ of types $J\boxtimes J \choose{J\, J}$ and
$V \choose {J\boxtimes J\, J}$ respectively such that
\begin{align*}
Y_e(j_1,z_1) \cY(j_2,z_2)j_3 = \cY^1(\cY^2(j_1,z_1-z_2)j_2,z_2)j_3.
\end{align*}
But now, since $J\boxtimes J \cong V$, we can get $\tilde{\cY^1}$ and
$\tilde{\cY^2}$ of types $V \choose{J\, J}$ and $V \choose {V\, J}$
respectively such that
\begin{align*}
Y_e(j_1,z_1) \cY(j_2,z_2)j_3 = 
\tilde{\cY^1}(\tilde{\cY^2}(j_1,z_1-z_2)j_2,z_2)j_3.
\end{align*}
Using the assumed fusion rules, $\tilde{\cY^2}$ is proportional to
$\cY$ and by \cite{FHL}, $\tilde{\cY^1}$ is proportional to the module
map $Y_J$. Therefore,
\begin{align*}
Y_e(j_1,z_1) \cY(j_2,z_2)j_3 = \lambda Y_e(Y_e(j_1,z_1-z_2)j_2,z_2)j_3.
\end{align*}
We must prove that the proportionality constant $\lambda=1$.

Let us also gather information about this proportionality constant in
all the 8 cases corresponding to each $u,v,w$ being in either $V$ or
$J$.  Let us temporarily grade the space $V\oplus J$ with $\Z_2$ such
that $V=(V\oplus J)^{0_t}$, $J=(V\oplus J)^{1_t}$. Here $t$ stands for
temporary.  This may or may not be the the intended $\Z_2$ grading as
in the definition of vertex operator superalgebra.  Using simplicity
of $V$ and $J$ and Theorem 11.9 of \cite{DL}, we know that arbitrary
products and iterates of $Y_e$ are non-zero.  We thus obtain a
well-defined map
$$F(g_1,g_2,g_3): \Z_2\times\Z_2 \times \Z_2 \rightarrow \C^\times,$$
which measures the failure in associativity. Our aim is to prove that
$F \equiv 1$.  Similarly, let
$$\Omega: \Z_2\times \Z_2\rightarrow \C^\times$$ 
denote the constants regarding skew-symmetry, i.e.,
$$Y_e(u,x)v=\Omega(i,j)e^{xL(-1)}Y_e(v,-x)u,$$
whenever $u \in (V\oplus J)^{i_t}$ and $v\in(V\oplus J)^{j_t}$ with $i,j\in\Z_2$.

At this point, we can proceed in two ways. One is by deriving and
using the equations satisfied by $F$ and $\Omega$ or the other way is to
proceed as in the proof of Theorem 4.1 of \cite{H6}.

For the relations satisfied by $F$ and $\Omega$, a general
derivation could be found in \cite{H2, Ch}. In our setup, one can also
refer to \cite{C}. With our setup, we have got a one-dimensional
$\cT$-commutativity datum with a normalized choice of intertwining
operators in the sense of \cite{C}. Therefore, using Lemmas 2.1.5,
2.1.7 and 2.2.7 in \cite{C}, we know that $(F, \Omega)$ is a
normalized abelian 3-cocycle on the group $\Z_2$ with coefficients in
$\C^\times$, i.e., for all $i,j,k,l\in\Z_2$ the following hold:

\begin{align}
F(0,i,j)=F(i,0,j)&=F(i,j,0)\\
\Omega(0,i)=\Omega(i,0)&=1\\
F(i,j,k)F(i,j+k,l)F(j,k,l)&=F(i+j,k,l)F(i,j,k+l)\label{eqn:pentagon}\\
F(i,j,k)^{-1}\Omega(i,j+k)F(j,k,i)^{-1}&=
\Omega(i,j)F(j,i,k)^{-1}\Omega(i,k)\label{eqn:hexagon1}\\
F(i,j,k)\Omega(i+j,k)F(k,i,j)&=\Omega(j,k)F(i,k,j)\Omega(i,k)
\label{eqn:hexagon2}
\end{align} 
Letting $i,j,k = 1$ in equation \ref{eqn:hexagon2} gives
\begin{align*}
F(1,1,1)\Omega(0,1)F(1,1,1)&=\Omega(1,1)F(1,1,1)\Omega(1,1)
\end{align*}
Hence,
\begin{align*}
F(1,1,1)=\Omega(1,1)^2.\label{eqn:FOmega}
\end{align*}
In our case, equation \ref{eqn:cJJskew} implies that $\Omega(1,1)$ is
equal to $c_{J,J}^{-1}$, thereby giving $F(1,1,1)=1$ and thereby
giving associativity of $Y_e$.  We conclude that the associativity of
$Y_e$ holds.  

An alternate and a more direct way to prove $F\equiv 1$ is as in the
proof of Theorem 4.1 of \cite{H6}.  The approach in \cite{H6} amounts
precisely to using equation \ref{eqn:hexagon2}.

Now, if $c_{J,J}=1$, we define a final $\Z_2$ grading on $X=V\oplus J$
by $X^0=X, X^1=0$  and if $c_{J,J}=-1$, $X^0=V, X^1=J$.

Recalling equation \ref{eqn:skewforY1}, we see that when $c_{J,J}=1$,
we get the skew-symmetry as in the definition of a
$\frac{1}{2}\Z$-graded vertex operator algebra and when $c_{J,J}=-1$,
we get the skew-symmetry as in the definition of a vertex operator
superalgebra. Now, we proceed exactly as in \cite{H6}.  Since $Y_e$
satisfies associativity and skew-symmetry, using results in \cite{H2},
one gets that $Y_e$ satisfies commutativity for $\frac{1}{2}\Z$-graded
vertex operator algebras when $c_{J,J}=1$ and satisfies commutativity
for vertex operator superalgebras when $c_{J,J}=-1$.  Since $Y_e$ has
only integer powers of the formal variable and hence rationality of
products and iterates holds.  The other axioms in the definition being
easy to verify, we conclude that $V\oplus J$ is a
$\frac{1}{2}\Z$-graded vertex operator algebra when $c_{J,J}=1$ and a
vertex operator superalgebra when $c_{J,J}=-1$.

\end{proof}

\begin{remark}
When $\theta_J=-1, c_{J,J}=1$, we get a vertex operator algebra with
``wrong statistics,'' that is, a genuinely $\frac{1}{2}\Z$-graded
vertex operator algebra.

When $\theta_J=1, c_{J,J}=-1$, we get a vertex operator superalgebra
with ``wrong statistics,'' that is, a $\Z$-graded vertex operator
superalgebra.
\end{remark}

\begin{remark}
Any two choices $j,j'$ of isomorphims $J\boxtimes J \cong V$ are
non-zero scalar multiples of each other since $V$ is simple. Let
$j' = \lambda j$, for some $\lambda\in\C^\times$.  
If we use $j'$ in place of $j$ in the proof
above, we get the map
\begin{align*}
Y_e'(v_1\oplus j_1,x)&(v_2\oplus j_2)  \\
= &(Y(v_1,x)v_2 + \lambda\cY(j_1,x)j_2)\oplus (Y_J(v_1,x)j_2 + e^{xL(-1)}Y_J(v_2,-x)j_1)
\end{align*}
Fix an $l\in\C^\times$ such that $l^2=\lambda$.  Define
$f:V\oplus J \rightarrow V\oplus J$ by $f(v\oplus w)=v\oplus lw$ for
$v\in V$, $w\in J$.  It is clear that $f$ is invertible and fixes the
vacuum and the conformal vector as they both belong to $V$.  Moreover,
\begin{align*}
f\left(Y_e'(v_1\oplus j_1,x)\right.&\left.(v_2\oplus j_2)\right)  \\
= &(Y(v_1,x)v_2 + \lambda\cY(j_1,x)j_2)\oplus l(Y_J(v_1,x) j_2 + e^{xL(-1)}Y_J(v_2,-x)j_1)\\
= &(Y(v_1,x)v_2 + \cY(lj_1,x)lj_2)\oplus (Y_J(v_1,x)l j_2 + e^{xL(-1)}Y_J(v_2,-x)lj_1)\\
=& Y_e(f(v_1\oplus j_1),x)(f(v_2\oplus j_2))
\end{align*}
We see that $f$ furnishes an isomorphism of the two vertex operator algebra 
structures on $V\oplus J$.
\end{remark}

\subsection{Beyond self-dual simple currents}

Combining our findings with those of Carnahan \cite{C} we now get several results on general simple current extensions of VOAs. 

\begin{thm}\label{thm:generalorderJ}
Let $J$ be a simple current.  This implies that $J\boxtimes J^{-1}\cong V$ is simple by \cite{CLR}.
Assume that $\theta_{J^k}=\pm \Id_{J^k}$, 
with $\theta_{J^{k+2}}$ having the same sign as $\theta_{J^k}$
all $k\in\Z$.  Then 
\[
V_e= \bigoplus_{j\in G}J^j
\]
 has a natural
structure of a strongly $G$-graded vertex operator superalgebra,
where $G$ is the cyclic group generated by $J$. (For the definition
of strongly graded, we refer the reader to \cite{HLZ}.) If $G$ is
finite, we get a vertex operator superalgebra.
\end{thm}

\begin{proof}

The assumptions on the twist combined with balancing imply that $c_{J^k,J^k}= \pm\Id$. 
By assumption, $\dim{ J^{i+j} \choose J^i\,J^j}=1$. For each $i,j$,
choose a non-zero logarithmic intertwining operator
$\tilde{\cY}_{i,j}\in{ J^{i+j} \choose J^i\,J^j}$, such that
$\tilde{\cY}_{0,j}$ is the vertex operator map. By simplicity of $J^k$s,
$L(0)$ acts semi-simply on each of the $J^k$s and hence
$\tilde{\cY}_{i,j}$ are free of logarithms by Remark 3.23 of \cite{HLZ}. Moreover, since
$\theta_{J^k}=\pm 1$, each $J^k$ is graded either by $\Z$ or by $1/2 + \Z$.
Combining with the rest of the asssumptions on $\theta_{J^k}$ we get that
each $\tilde{\cY}_{i,j}$ has only integer powers of the formal variable.

Since each $J^i$ is simple, it is be easy to see, using Theorem
11.9 of \cite{DL} that the products and iterates of
$\tilde{\cY}_{i,j}$ are non-zero. This is needed to ensure that
``$F$,'' defined below, is well-defined and non-zero.  

We know, using Proposition 3.44 of \cite{HLZ} that
$e^{xL(-1)}\tilde{\cY}_{i,j}(u,e^{\pi i}x)v =
e^{xL(-1)}\tilde{\cY}_{i,j}(u,-x)v$
is a (non-zero) intertwining operator of type
$J^{i+j}\choose {J^j\, J^i}.$ Also, since
$\dim{ J^{i+j} \choose J^j\,J^i}=1$, we know that it must be a
(non-zero) scalar multiple of $\tilde{\cY}_{j,i}$.  So, we get a
$\cT$-commutativity datum in the sense of \cite{C} along with a
normalized choice of intertwining operators $\tilde{\cY}$.  With this,
we get a normalized abelian cocycle $(\tilde{F},\tilde{\Omega})$ on
the group $G=\{J^k\,|\,k\in\Z\}$.

As before, $\tilde{\Omega}(i,j)= c_{J^i,J^j}^{-1}$, and hence,
$\tilde{\Omega}(i,i)=\pm 1$.  Using Lemmas 1.2.7 and 1.2.8 of
\cite{C}, we get that $\tilde{\Omega}(2i,2i)=1$ for $i\in G$ and that
$(\tilde{F},\tilde{\Omega})$ is cohomologous to a normalized abelian
3-cocycle $(F,\Omega)$ on $G$, pulled back from a normalized abelian
cocycle $(\bar{F},\bar{\Omega})$ for the abelian group $G/2 G$.  We
now modify our choice of $\{\tilde{\cY}_{i,j}\}$ (by multiplying each
intertwining operator with an appropriate scalar) so as to match with
$(F,\Omega)$. This is possible by Theorem 2.2.13 (i) of \cite{C}.

So, we get, using Proposition 2.4.5 of \cite{C} that
$\tilde{V}=\bigoplus_{i\in 2G}J^i$ is a $\Z$-graded vertex
operator algebra (if $|G|=\infty$, we would get a strongly $G$-graded
vertex operator algebra), that
$\tilde{J}=\bigoplus_{i \not\in 2G}J^i$ is a $\tilde{V}$-module (or a
strongly $G$-graded $\tilde{V}$-module if $|\{i \not\in 2G\}|=\infty$),
and a $\tilde{V}$-intertwining operator
$\cY\in{\tilde{V}\choose \tilde{J}\,\tilde{J}}$.
Therefore, we are done if $G=2G$.

Now assume that $G\neq 2G$. Note that since $G$ is cyclic, $G/2G\cong \Z_2$.
Let $|z_1|>|z_2|>|z_1-z_2|>0$ and
$w_{i_t}\in J^{i_t}$ for $t\in\{1,2,3\}, i_t\in G\backslash 2\cdot G$.
Note that by assumption, we have:
\begin{align*}
\cY_{i_1,i_2}(w_{i_1},z_1)w_{i_2} &= 
\Omega(i_1,i_2)e^{-z_1L(-1)}\cY_{i_2,i_1}(w_{i_2},-z_1)w_{i_1}\\
\cY_{i_1,i_2+i_3}(w_{i_1},z_1)\cY_{i_2,i_3}(w_{i_2},z_2)w_{i_3} &= 
F(i_1,i_2,i_3)
\cY_{i_1+i_2,i_3}(\cY_{i_1,i_2}(w_{i_1},z_1-z_2)w_{i_2},z_2)w_{i_3}.
\end{align*}
However, using that $(F,\Omega)$ is a pull-back of
$(\bar{F},\bar{\Omega})$ on $G/2G$ we get:
\begin{align*}
\cY(w_{i_1},z_1)w_{i_2} &= 
\bar{\Omega}(\bar{i_1},\bar{i_2})e^{-z_1L(-1)}\cY(w_{i_2},-z_1)w_{i_1}\\
\cY(w_{i_1},z_1)\cY(w_{i_2},z_2)w_{i_3} &= 
\bar{F}(\bar{i_1},\bar{i_2},\bar{i_3})
\cY(\cY(w_{i_1},z_1-z_2)w_{i_2},z_2)w_{i_3},
\end{align*}
here $\bar{i}$ denotes the image of $i$ in $\Z/2Z$.
Moreover, recall that  $\Omega(i,i)=\pm 1$ for all $i\in G$
and hence, $\bar{\Omega}(i,i)=\pm 1$ for all $i\in G/2G \cong \Z_2$.
In other words, $\bar{\Omega}(i,i)^2=1$ for all $i \in G/2G \cong \Z_2$.
From the proof of Theorem \ref{thm:constructingVOSA} it is clear that
$\bar{F}\equiv 1$.
Now we proceed as in Theorem \ref{thm:constructingVOSA} to prove that
we indeed get a vertex operator (super) algebra structure on
$\tilde{V}\oplus \tilde{J}$.  If $|G| = \infty$ then we actually get a
strongly graded vertex operator (super) algebra with abelian group $G$
(cf.\ \cite{HLZ}).
\end{proof}

\begin{thm}\label{thm:HKLanalog}
Let $V$ be a vertex operator algebra such that its module category
$\cC$ has a natural vertex tensor category structure in the sense of
Huang-Lepowsky.  Then, the following are equivalent:
\begin{itemize}
\item A vertex operator superalgebra $V_e$ such that 
$V$ is a subalgebra of  $V_e^0$.
\item A $\cC$-superalgebra $V_e$ with $\theta^2=\Id_{V_e}$.
\end{itemize}
\end{thm}

\begin{proof}
This proof is almost the same as the proof in the \cite{HKL}.  Here,
we only give those details that are different from the ones in
\cite{HKL}.

\noindent 
(i) We first prove that a vertex superalgebra yields a $\cC$-superalgebra.

Let $V_e$ be a vertex operator superalgebra such that $V$ is a
subalgebra of $V_e^0$. We immediately get a morphism
$\iota:V\hookrightarrow V_e^0$.  Being a $V$-module, $V_e$ is an
object of $\cC$.  Even in the case that $V_e$ is a superalgebra, since
$V\subset V_e^{0}$, $Y_{V_e}$ is an intertwining operator for
$V$-modules of the type $V_e\choose{V_e\,V_e}$. By the universal
property of $V_e\boxtimes V_e$, there exists a unique module map
$\mu : V_e\boxtimes V_e \rightarrow V_e$ such that
$\overline{\mu}\circ\boxtimes = Y_e(\cdot,1)\cdot$.

We now prove that $(V_e, \mu,\iota)$ along with its $\frac{1}{2}\Z$
and $\Z_2$-gradings is a $\cC$-superalgebra.

Clearly, since $Y_e$ only has integral powers of the formal variable,
\begin{align}
\overline{\mu}(\theta u \boxtimes \theta v) 
&= Y_e(\theta u, 1)\theta v\nonumber\\
&= Y_e(e^{2i\pi L(0)} u, 1)e^{2i\pi L(0)} v\nonumber\\ 
&= e^{2i\pi L(0)}Y_e( u,e^{2i\pi}1)v\nonumber\\
&= e^{2i\pi L(0)}Y_e( u, 1) v\nonumber\\
&= \overline\theta Y_e( u, 1) v\nonumber\\
&=\overline\theta \overline\mu(u\boxtimes v)
\end{align}

For $i,j\in \{0,1\}$ and $u\in V_e^i, v\in V_e^j$,
$Y_e(u,x)v \in V_e^{i+j}((x))$, and hence, $\mu$ respects the
$\Z_2$-grading as well.

The proof of associativity of $\mu$ given in \cite{HKL} goes through
line-by-line.

We now turn to skew-symmetry and the relation of $\mu$ to braiding
$\cR$.  The braiding isomorphism $\cR$ is determined uniquely by
\begin{align*}
\overline\cR(u\boxtimes v) = e^{L(-1)}\overline{\cT}_{\gamma_1^-}(v\boxtimes_{P(-1)}u)
\end{align*}
where $u,v\in V_e$ and $\gamma_1^-$ is a clockwise path from $-1$ to
$1$ in $\overline{\mathbb{H}}\backslash \{0\}$, and $\cT_{\gamma_1^-}$
is the corresponding parallel transport isomorphism
$V_e\boxtimes_{P(-1)}V_e \rightarrow V_e\boxtimes V_e$.  Therefore,
\begin{align*}
\overline\mu(\overline\cR(u\boxtimes v))=
\overline\mu(e^{L(-1)}\cT_{\gamma_1^-}(v\boxtimes_{P(-1)}u)).
\end{align*}
Let $\cY$ be the intertwining operator of type
$V_e\boxtimes V_e \choose {V_e\,V_e}$ corresponding to the
intertwining map $\boxtimes_{P(1)}$.  Then,
\begin{align*}
e^{L(-1)}\cT_{\gamma_1^-}(v\boxtimes_{P(-1)}u)=e^{L(-1)}\cY(v,e^{i\pi})u.
\end{align*}
For $\Z_2$-homogeneous elements $u\in V_e^i$, $v\in V_e^j$, for 
$i,j\in\Z/2\Z$, $Y_e$ 
satisfies the skew-symmetry:
\begin{align}
Y_e(u,1)v = (-1)^{ij}e^{L(-1)}Y_e(v,-1)u.
\end{align}
Therefore, we have that:
\allowdisplaybreaks
\begin{align}
(-1)^{ij}\overline\mu(\overline\cR(u\boxtimes v))
&=(-1)^{ij}\overline\mu(e^{L(-1)}\cT_{\gamma_1^-}(v\boxtimes_{P(-1)}u))\nonumber\\
&=(-1)^{ij}\overline\mu(e^{L(-1)}\cY(v,e^{i\pi})u)\nonumber\\
&=(-1)^{ij}e^{L(-1)}\overline\mu(\cY(v,e^{i\pi})u)\nonumber\\
&=(-1)^{ij}e^{L(-1)}Y_e(v,e^{i\pi})u\nonumber\\
&=(-1)^{ij}e^{L(-1)}Y_e(v,-1)u\nonumber\\
&=Y_e(u,1)v\nonumber\\
&=\overline\mu(u\boxtimes v)
\end{align}

The left unit isomorphism property also goes through exactly as in
\cite{HKL}.

\noindent (ii) Now we prove that a $\cC$-superalgebra yields a vertex
operator superalgebra.

We are given a $\cC$-superalgebra $(V_e,\mu,\iota)$.  By definition,
$V_e$ is a generalized $V$-module. The $P(1)$-intertwining map
$\overline\mu\circ\boxtimes_{P(1)}: V_e\otimes V_e\rightarrow
\overline{V_e}$
corresponds to an intertwining operator $Y_e$ of type
$V_e \choose {V_e\,V_e}$ such that
$$\overline\mu(u\boxtimes v)=Y_e(u,1)v.$$
Using $\iota$, we can view vacuum vector $\mathbf{1}$ and conformal
vector $\omega$ as elements of $V_e$.

The Virasoro relations and the fact that $V_e$ is graded by
generalized eigenvalues of $L(0)$ follow from the fact that $V_e$ is a
generalized $V$-module.  $L(-1)$-derivative property and the
$L(0)$-conjugation formula for $Y_e$ follow from the fact that $Y_e$
is an intertwining operator.

Since $\theta^2=\Id_{V_e}$, $L(0)$ acts semisimply on $V_e$ and $V_e$ is
in fact $\frac{1}{2}\Z$-graded by eigenvalues of $L(0)$.  Since $L(0)$
acts semisimply, $Y_e$ does not have logarithms.  Using
$\mu(\theta \otimes \theta) = \theta \circ \mu$, the definition of
$\theta$, the $L(-1)$-derivative property of $Y_e$ and the
$L(0)$-conjugation formula for $Y_e$, we get that
\begin{align*}
\overline\mu(\theta u\boxtimes \theta v) &= Y_e(e^{2i\pi L(0)}u,1)e^{2i\pi L(0)}v
=e^{2i\pi L(0)}Y_e(u,e^{2i\pi}1)v,\\
\overline\theta\overline\mu(u\boxtimes v)&=e^{2i\pi L(0)}Y_e(u,1)v
\end{align*}
and hence, 
\begin{align*}
Y_e(u,e^{2i\pi}1)v=Y_e(u,1)v.
\end{align*}
Letting $u$ and $v$ to be homogeneous with respect to the
$L(0)$-grading, we immediately conclude that $Y_e(u,x)v$ must have
only integral powers of $x$.

Skew-symmetry with the correct factor of $-1$, vacuum property,
creation property and associativity follow exactly as in \cite{HKL}.
Using \cite{H2}, skew-symmetry and the associativity imply the right
kind of commutativity for a vertex operator superalgebra. Again using
results in \cite{H2}, commutativity and associativity yield the right
kind of rationality of products and iterates, these yield the desired
Jacobi identity.

\end{proof}

\begin{thm}\label{thm:modules}
Let $V,V_e,\cC$ be as in Theorem \ref{thm:HKLanalog}.
The category of generalized modules for the vertex operator superalgebra $V_e$
is isomorphic to the category $\rep^0 V_e$, where $V_e$ is
considered as a $\cC$-superalgebra.
\end{thm}
\begin{proof}
It is clear that a generalized $V_e$-module corresponds to an object in
$\rep^0 V_e$. 

Let $(W,\mu_W)\in\rep^0 V_e$. Let $Y_W$ be an intertwining operator
corresponding to the intertwining map 
$\overline{\mu_W}\circ \boxtimes: A\otimes W\rightarrow \overline{W}$,
and let $\cY_\boxtimes$ be the intertwining operator corresponding to 
the intertwining map $\boxtimes:A\otimes W\rightarrow \overline{A\boxtimes W}$.
We have that $Y_W=\mu_W\circ\cY_\boxtimes$.
Let us analyze the condition $\mu_W\circ(\cR_{W,A}\cR_{A,W}) = \mu_W$.
This says that 
$$\overline{\mu_W}\circ(\cY_\boxtimes(v,e^{2\pi i})w)=\overline{\mu_W}(\cY_\boxtimes(v,1)w),$$
for $v\in V_e, w\in W$.
Therefore,
\begin{align}
\label{eqn:YWmon}
Y_W(v,e^{2\pi i})w=Y_W(v,1)w.
\end{align}
Using the notation (3.24) of \cite{HLZ} let
$$Y_W(v,x)w=\sum\limits_{n\in\C,k \in \N}\, (v_{n;k}^{Y_W}w)\, x^n (\log x)^k,$$
so that for a complex number $\zeta$,
$$Y_W(v,e^\zeta)w=\sum\limits_{n\in\C,k \in \N}\, (v_{n;k}^{Y_W}w)\, e^{n \zeta } \zeta^k.$$
Equation \eqref{eqn:YWmon}
now gives that
$$\sum\limits_{n\in\C,k \in \N}\, (v_{n;k}^{Y_W}w)\, e^{2\pi i n} (2\pi i)^k
= \sum\limits_{n\in\C}\, v_{n;0}^{Y_W}w ,$$
which in turn immediately implies that $Y_W$ has no logarithms and only
integral powers of the formal variable.
Now, the proof of Theorem 3.4 of \cite{HKL} goes through.
\end{proof}

\begin{thm}\label{thm:liftsiff}
Consider the set-up of Theorem \ref{thm:generalorderJ}.
Let $V_e=\bigoplus_{j\in G}J^j$.
If $X\in\cC$ 
then, $\cF(X)\in\rep^0 V_e$ if and only if $M_{J,X}=\Id_{J\boxtimes X}$.
\end{thm}
\begin{proof}
We assume that $X\neq 0$.

Observe that $V_e\boxtimes X \cong \bigoplus_j J^j\boxtimes X$ by
Proposition 4.24 of \cite{HLZ}.

Let us first do the only if part. 
Since $\cF(X)\in\rep^0 V_e$, 
$$
\mu_{\cF(X)}|_{J\boxtimes (J^0\boxtimes X)}\circ 
M_{J,J^0\boxtimes X}=\mu_{\cF(X)}|_{J\boxtimes (J^0\boxtimes X)}.
$$
However, by definition of $\mu_{\cF(X)}$,
$\mu_{\cF(X)}|_{J\boxtimes(J^0\boxtimes X)}=(\mu_{J\boxtimes
J^0}\otimes \Id_X)\circ\cA_{J,J^0,X}$.
Now, $J\boxtimes J^0 = J\boxtimes V$ is spanned by homogeneous weight
components of $j\boxtimes v$ as $j$ runs over $J$ and $v$ runs over
$V$ (cf.\ Proposition 4.23 of \cite{HLZ}).  By definition,
$\overline{\mu}(j\boxtimes v) = e^{L(-1)}Y_J(v,-1)j$, where $Y_J$ is
the module map corresponding to $J$.  Hence, $\mu|_{J^1\boxtimes J^0}$
is non-zero. Since $J$ is simple, $\mu|_{J^1\boxtimes J^0}$ is a
morphism of simple modules, and hence, being non-zero, is invertible.
Therefore, $\mu_{\cF(X)}|_{J\boxtimes(J^0\boxtimes X)}$ is invertible also.
We conclude that $M_{J,J^0\boxtimes X}$ must be identity.

Conversely, by assumption, $c_{J,J}=\pm\Id$ and hence $M_{J,J}=\Id$.
Moreover, if $M_{J,X}=\Id$ then by Theorem \ref{thm:monodromy},
$M_{J^i,J^j\boxtimes X}=\Id$ and hence $\cF(X)\in\rep^0 V_e$.
\end{proof}

\begin{cor}\label{cor:liftingsimple}
If $X$ is a simple $V$-module, then $\cF(X)$ is a $V_e$-module iff 
$h_{J\boxtimes X}-h_{J}-h_{X}\in\Z$, where $h_{\bullet}$ denotes
the conformal dimension.
\end{cor}
\begin{proof}
Recall that $M_{J,X}=\theta_{J\boxtimes X}\circ(\theta_J^{-1}\boxtimes\theta^{-1}_X)$
and that $\theta=e^{2\pi iL(0)}$. Since $J$, $X$ and $J\boxtimes X$
are simple, $L(0)$ acts semisimply on $J$, $X$  and $J\boxtimes X$.
Hence, $M_{J,X}=e^{2\pi i (h_{J\boxtimes X}-h_J-h_X)}\Id_{J\boxtimes X}$.
Now use Theorem \ref{thm:liftsiff}.

\end{proof}

Now we focus on the case when $P$ is an indecomposable object.
In what follows, the point is to prove that under certain conditions,
the (locally) nilpotent part of the monodromy $M_{J,P}$ vanishes.
\begin{lemma}\label{lem:MThetaIndecomp}
Let $J$ be a finite order simple current
as in Theorem \ref{thm:generalorderJ}.  Let $P$ be an
indecomposable object with $\dim(\Hom(P,P))<\infty$ and
$\dim(\Hom(J\boxtimes P, J\boxtimes P))<\infty$.  Assume  that
$L(0)$ has Jordan blocks of bounded size on both $P$ and
$J\boxtimes P$.  Then,
$M_{J,P}=(\theta_{J\boxtimes P})_{ss}\circ(\theta^{-1}_J\boxtimes
(\theta^{-1}_P)_{ss})$,
where $\bullet_{ss}$ denotes the semi-simple part.
In particular, $M_{J,P}=\lambda\Id_{J\boxtimes P}$ for some 
$\lambda\in\C$. Moreover, $\lambda^N=1$, where $N$ is the order
of $J$.
\end{lemma}
\begin{proof}
Note the very important property of simple currents from a
forthcoming article \cite{CLR} that $\bullet\boxtimes J$ and
$J\boxtimes \bullet$ are exact for any invertible simple current
$J$. Moreover, it is easy to see that these two functors take
non-zero objects to non-zero objects.  Using these and the fact that
$P$ is indecomposable, one can prove that $J\boxtimes P$ is
indecomposable.  Indeed, if we have a split short exact sequence
\begin{align*}
\xymatrix{
0  \ar[r] & A \ar[r] & J\boxtimes P \ar[r] & B \ar@{.>}@/^1pc/[l] \ar[r]& 0}
\end{align*}
we get a split short exact sequence 
\begin{align*}
\xymatrix{
0  \ar[r] & J^{-1}\boxtimes A \ar[r] & J^{-1}\boxtimes(J\boxtimes P) \ar[r] & 
J^{-1}\boxtimes B \ar@{.>}@/^1.3pc/[l] \ar[r]& 0.
}
\end{align*}
But, since $J^{-1}\boxtimes(J\boxtimes P)$ and $P$ are isomorphic via
associativity and the property of unit,
$J^{-1}\boxtimes(J\boxtimes P)$ is indecomposable as well.

Define $\theta_{ss}$ to be the semi-simple part of $\theta$ and
$\theta_{nil}$ to be $\theta-\theta_{ss}$.  Since twist given by
$\theta = e^{2\pi i L(0)}$ and since $L(0)$ has Jordan blocks of
bounded size on both $P$ and $J\boxtimes P$, we indeed have that some
finite positive power of $\theta_{nil}$ is $0$.  Note that since $J$
is simple, $\theta$ acts semi-simply on $J$.  We have that
\begin{align}
M_{J,P}&=\theta_{J\boxtimes P}\circ(\theta^{-1}_J\boxtimes\theta^{-1}_P)
\nonumber\\
&=(\theta_{J\boxtimes P})_{ss}\circ(\theta^{-1}_J\boxtimes(\theta^{-1}_P)_{ss})
+(\theta_{J\boxtimes P})_{ss}\circ(\theta^{-1}_J\boxtimes(\theta^{-1}_P)_{nil})
\nonumber \\& \quad 
+(\theta_{J\boxtimes P})_{nil}\circ(\theta^{-1}_J\boxtimes(\theta^{-1}_P)_{ss})
+(\theta_{J\boxtimes P})_{nil}\circ(\theta^{-1}_J\boxtimes(\theta^{-1}_P)_{nil}).
\label{eqn:monodromyssnil}
\end{align}
Let, 
\begin{align*}
(M_{J,P})_{ss}
&=(\theta_{J\boxtimes P})_{ss}\circ(\theta^{-1}_J\boxtimes(\theta^{-1}_P)_{ss}),
\end{align*}
and observe that $(M_{J,P})_{ss}$ is indeed semi-simple.

Let $(M_{J,P})_{nil}=M_{J,P} - (M_{J,P})_{ss}$. We now prove that
$(M_{J,P})_{nil}$ is indeed nilpotent.  By definition, morphisms in
our category commute with $L(0)$, and hence commute with $L(0)_{ss}$
and $L(0)_{nil}$. Therefore, $\theta_{ss}$ and $\theta_{nil}$ are
natural. Hence,
\begin{align*}
(&(M_{J,P})_{nil})^K   \\
&=\sum\limits_{\substack{a+b+c = K\\a,b,c\in\N}}C_{a,b,c} 
((\theta_{J\boxtimes P})_{ss}(\theta^{-1}_J\boxtimes(\theta^{-1}_P)_{nil}))^a
((\theta_{J\boxtimes P})_{nil}(\theta^{-1}_J\boxtimes(\theta^{-1}_P)_{ss}))^b
((\theta_{J\boxtimes P})_{nil}(\theta^{-1}_J\boxtimes(\theta^{-1}_P)_{nil}))^c
\end{align*}
for some constants $C_{a,b,c}.$ However, since Jordan blocks of
$L(0)$ on $P$ and $J\boxtimes P$ are bounded in size, one can now pick
a large enough $K$ for which $((M_{J,P})_{nil})^K=0$.

It is easily seen that for any indecomposable module $X$, all the
generalized eigenvalues of $L(0)$ on $X$ belong to a single coset
$\mu + \Z$ where $\mu\in\C$, and hence, $(\theta_{J\boxtimes P})_{ss}$
and $(\theta_P^{-1})_{ss}$ are scalar multiplies of identity.  Hence,
we deduce that $(M_{J,P})_{ss}=\lambda\Id_{J\boxtimes P}$ for some
scalar $\lambda\in\C$.

Combining everything, we get that
$M_{J,P}=\lambda \Id_{J\boxtimes P}+ \pi$ where $\pi$ is some
nilpotent endomorphism of $J\boxtimes P$ and $\lambda\in\C$. Now use
Lemma \ref{lem:JPmon}.
\end{proof}

Combining with Theorem \ref{thm:liftsiff}, we arrive at the following
criterion for lifting indecomposable objects.

\begin{thm}\label{cor:liftingprojective}
Let $J$ be a finite order simple current as in Theorem \ref{thm:generalorderJ}.
Let $P$ be an indecomposable object with finite dimensional endomorphism ring, 
$\dim(\Hom(P,P))<\infty$ and also
$\dim(\Hom(J\boxtimes P, J\boxtimes P))<\infty$.  Assume also that
$L(0)$ has Jordan blocks of bounded size on both $P$ and
$J\boxtimes P$.  Then, $\cF(P)$ is a generalized $V_e$-module iff
$h_{J\boxtimes P}-h_{J}-h_{P}\in\Z$.
\end{thm}

Now we give some useful criteria for the cases when $J$ does not
necessarily have finite order.

\begin{lemma}\label{lem:infordermonodromy}
Let $J$ be a simple current. Let $A_i$ and $Q$ be objects in $\cC$ such that 
$M_{J,A_i}=\lambda_i\Id_{A_i}$ and $M_{J,Q}=\lambda\Id_{J,Q}$ for some scalars $\lambda_i$ and
$\lambda$. Then the following hold.

\begin{enumerate}
\item $M_{J,\boxtimes_{i=1}^NA_i}=\Big(\prod\limits_{i=1}^N\lambda_i\Big)\Id_{\boxtimes_{i=1}^N A_i}$.
\item If $0\rightarrow A \rightarrow Q\rightarrow B\rightarrow 0$ is a short
exact sequence of generalized $V$-modules then $M_{J,A}=\lambda\Id_{J\boxtimes A}$ and $M_{J,B}=\lambda\Id_{J\boxtimes B}$.
\item If $P$ is a subquotient of $Q$ then
  $M_{J,P}=\lambda\Id_{J\boxtimes P}$.
\end{enumerate}
\end{lemma}
\begin{proof}
(1) follows by retracing the proof of equation \eqref{eqn:monA-BC}
in Theorem \ref{thm:monodromy} and then using induction on $i$.
For (2), by exactness of $J\boxtimes \bullet$ from \cite{CLR}, and by naturality
of monodromy (Lemma \ref{lem:MisNatural}), we get the following commutative diagram.
\begin{align*}
\xymatrix{
0 \ar[r] & J\boxtimes A \ar[r] \ar[d]_{M_{J,A}} & J\boxtimes Q \ar[r] \ar[d]_{M_{J,Q}}
& J\boxtimes B \ar[d]_{M_{J,B}}\ar[r] & 0 \\
0 \ar[r] & J\boxtimes A \ar[r]  & J\boxtimes Q \ar[r] 
& J\boxtimes B \ar[r] & 0.
}
\end{align*}
With this, (2) follows.
Now, (3) follows from (2).
\end{proof}

Since $M=\theta\circ(\theta^{-1}\boxtimes\theta^{-1})$, we see that
Lemma \ref{lem:infordermonodromy} holds if $A_i$ are simple $V$-modules.
We immediately get the following Theorem.
\begin{thm}\label{cor:inforderJlifting} Let $J$ be simple current as
in Theorem \ref{thm:generalorderJ}.  $J$ need not necessarily have
finite order. Let $P$ be a subquotient of $\boxtimes_{i=1}^N A_i$ 
for some simple $V$-modules $A_i$.
Then, $\cF(P)$ is a
generalized $V_e$-module iff $h_{J\boxtimes P}-h_J-h_P\in\Z$.
\end{thm}
\begin{proof}
Using Lemma \ref{lem:infordermonodromy}, we see that
$M_{J,P}=\lambda\Id_{J\boxtimes P}$ for some scalar $\lambda$. Using
this, using equation \eqref{eqn:monodromyssnil} and observing that
$\theta_{nil}$ and $(\theta^{-1})_{nil}$ are locally nilpotent, we get that
$M_{J,P}=(\theta_{J\boxtimes P})_{ss}\circ((\theta_{J})^{-1}\boxtimes
(\theta_P^{-1})_{ss})$. Now the conclusion follows.
\end{proof}

\begin{remark}\label{rem:Fisexact}
In Section \ref{sec:examples}, we will need the fact that the
induction functor $\cF$ is exact, in order to deduce the Loewy
diagrams of induced modules. In our setup, one can use the fact from
\cite{CLR} that $J\boxtimes \bullet$ is an exact functor for a
simple current $J$ to deduce that $\cF$ is exact. Alternately,
one can proceed as in Theorem 1.6 of \cite{KO}.
\end{remark}

We now summarize Huang's theorem \cite{H5} on when the Huang-Lepowsky-Zhang
theory can be applied.

\begin{thm} (Cf.\ \cite{H5}.)  \label{thm:cond1}
\begin{itemize}
\item Let $V$ be such that (1) $V$ is
  $C_1^a$-cofinite, i.e.,
  $\text{Span}\{ u_nv, L(-1)u\,|\, u,v\in V_+\}$ has finite
  codimension, (2) There exists a positive integer $N$ such that the
  differences between the real parts of the lowest conformal weights
  of irreducible $V$-modules are bounded by $N$ and such that the
  associative algebra $A_N(V)$ (cf.\ \cite{DLM3}) is finite dimensional and (3)
  Irreducible $V$-modules are $\R$-graded and $C_1$-cofinite.  Then
  the category of grading-restricted generalized $V$-modules
  (i.e., lower truncated modules with finite dimensional generalized
  weight spaces) satisfies
  the conditions required to invoke Huang-Lepowsky-Zhang's theory.

\item If $V$ is a $C_2$-cofinite such that $V_n=0$ for $n<0$ and
  $V_0 =\C\mathbf{1}$ then all the three conditions mentioned above
  are satisfied and the category of generalized grading-restricted
  modules of $V$ has a natural vertex tensor category structure, in
  particular, it has a braided tensor category structure.
\end{itemize}
\end{thm}

We can also combine the main result of \cite{Miy}
with Theorems 12.15 and 12.16 of \cite{HLZ} to obtain the following.

\begin{thm}\label{thm:cond2}
Let $V$ be a vertex operator algebra and consider a full sub-category
$\cC$ of generalized modules of $V$ such that:
\begin{enumerate}
\item $\cC$ is abelian.
\item For each object $\cC$, the generalized weights are real numbers,
there is a $K\in\Z_+$ such that $(L(0)-L(0)_{ss})^K=0$, where $L(0)_{ss}$
is the semi-simple part of $L(0)$.
\item $\cC$ is closed under images, contragredients, taking finite direct
sums and $V$ is an object of $\cC$.
\item Every object of $\cC$ satisfies $C_1$-cofiniteness
($\text{span}\{u_{-1}w|u\in V_+,w\in W\}$ has finite codimension in $W$),
has finite dimensional generalized weight spaces with lower truncated weights
and is quasi-finite dimensional ($\bigoplus_{n<N}W_{[n]}$ is finite dimensional
for any $N\in\R$, where $W_{[n]}$ denotes the generalized eigenspace for $L(0)$
with generalized eigenvalue $n$.)
\end{enumerate}
Then $\cC$ has a natural vertex tensor category structure,
in particular, it has a braided tensor category structure.
\end{thm}
\begin{proof}
The Main Theorem of \cite{Miy} ensures that $\cC$ is closed
under $\boxtimes$. The rest follows from Theorems 12.15 and 12.16 of \cite{HLZ}.
\end{proof}

\section{Building logarithmic VOAs from parts}
\label{sec:examples}

As Carnahan suggested, one can now explicitly build
various VOAs from parts. Our interest is in the non semi-simple, i.e.,
logarithmic type and here we will construct a few examples. 

Recall the definition and construction of the contragredient modules
from Section 5.2 of \cite{FHL} and Theorem 2.34 of \cite{HLZ}.  Recall
the skew-symmetry and the adjoint operations on intertwining
operators, equations (3.77) and (3.87), respectively, and the
corresponding Propositions 3.44 and 3.46, respectively, from
\cite{HLZ}.  We use the same notations for contragredients
($\bullet'$), skew-symmetry ($\Omega_r(\bullet)$) and the adjoint
($\cA_r(\bullet)$) operations as in \cite{HLZ}.  Also recall from
\cite{HLZ} Definition 4.29 of a finitely reductive vertex operator
algebra.

Let $V$ be a finitely reductive simple vertex operator algebra such
that $V$ is isomorphic to its contragredient, i.e., $V\cong V'$.  For
such a $V$, by Theorem 4.33 of \cite{HLZ}, the category of $V$-modules
is closed under $\boxtimes_{P(z)}$ tensor products. Let $J$ be a
simple current for $V$.

Using Propositions 3.44 and 3.46 of \cite{HLZ} and the assumption that
$V\cong V'$, we know that the fusion rules
$\cN^J_{V,J}=\cN^J_{J,V}=\cN^{V'}_{J,J'}=\cN^V_{J,J'}=\cN^V_{J',J}$
are non-zero.  Therefore, by the universal property of the tensor
products, there exists a (non-zero) morphism
$J\boxtimes J'\rightarrow V$.  This means that $V$ is in fact a direct
summand of $J\boxtimes J'$ because $V$ is simple and finitely
reductive.  Since $J$ is simple, $J'$ is simple. Moreover, since $J$
is a simple current, $J\boxtimes J'$ is simple as well. Hence,
$J\boxtimes J'\cong V$.  Hence,
$$J'\cong (J^{-1}\boxtimes J)\boxtimes J'
\cong J^{-1}\boxtimes(J\boxtimes J')\cong J^{-1}\boxtimes V\cong J^{-1}.$$
In the particular case that $J$ is a self-dual simple current, i.e.,
$J\boxtimes J\cong V$, we indeed get that $J\cong J^{-1}\cong J'$.

The following proposition will be used later in the case when $V$ is a
simple finitely reductive vertex operator algebra such that
$V'\cong V$, $V_n=0$ for $n<0$ and $J$ is a self-dual simple
current.  However, we state the proposition in the most general
setting.

\begin{prop}\label{prop:J-OPE}
Let $V$ be vertex operator algebra and $J$ be a (non-zero)
$V$-module.  There exists an intertwining operator $\cY$ of type
$V' \choose J\, J'$ and elements $j\in J$ and $j'\in J'$ of lowest
conformal weight, say $d$, such that
\begin{align*}
\langle \cY(j,x)j',\one\rangle \neq 0.
\end{align*}
Moreover, if $V$ is such that $V_n=0$ for $n<0$, then, there exists
a non-zero $v'\in V'$ of conformal weight $0$ such that 
\begin{align}
\cY(j,x)j' =x^{-2d}v' + \cdots. 
\label{eq:non-degeneracy-useful}
\end{align}
\end{prop}
\begin{proof}
Let $j\in J$ be any (non-zero) vector of lowest conformal dimension
$d$, and let $j'\in J'$ be of the same conformal dimension such that
$\langle j', j\rangle\neq 0$. Let $Y$ be the module map corresponding
to $J$. Pick $r,s\in\Z$ and let $\cY$ be the (non-zero)
intertwining operator of type ${V'\choose J\, J'}$ given by
$\cA_{r}(\Omega_{s}(Y))$.  Being a module map, $Y$
has no monodromy and therefore $\Omega_{s}(Y)$ has no monodromy and
$\cY$ is independent of $s$.  We have:
\begin{align}
\langle  \cY(j,x)j', \one\rangle 
&= 
\langle  \cA_{r}(\Omega_{s}(Y))(j,x)j',\one \rangle 
\nonumber \\
&= \langle j',\Omega_{s}(Y)(e^{xL(1)}e^{(2r+1)\pi iL(0)}(x^{-L(0)})^2j,x^{-1})
 \one \rangle
\nonumber \\
&= \langle j',\Omega_{s}(Y)(e^{(2r+1)\pi iL(0)}(x^{-L(0)})^2j,x^{-1})
 \one \rangle
\nonumber \\
&= e^{d(2r+1)\pi i}x^{-2d}\langle j',\Omega_{s}(Y)(j,x^{-1}) \one \rangle
\nonumber \\
&= e^{d(2r+1)\pi i}x^{-2d}\langle j', e^{x^{-1}L(-1)}Y(\one ,-x^{-1}) j  \rangle
\nonumber \\
&= e^{d(2r+1)\pi i}x^{-2d}\langle e^{x^{-1}L(1)}j', j  \rangle
\nonumber \\
&= e^{d(2r+1)\pi i}x^{-2d}\langle j',j \rangle \neq 0.
\label{eq:non-degeneracy}
\end{align}
Moreover, if $V$ is such that $V_n=0$ for all $n<0$, then
$(V')_n=0$ for all $n<0$ which means that $\cY(j,x)j'$ does not
have any powers of $x$ lower than $x^{-2d}$ and 
equation \eqref{eq:non-degeneracy-useful} follows.
\end{proof}

The following proposition  will be used later  to calculate
quantum dimensions of certain simple currents.
\begin{prop}\label{prop:qdimsimplecurrent}
Let $J$ be a self-dual simple current. 
Let $j\in J$ be a non-zero element of lowest conformal weight, say $d$.
Then, 
\begin{align}
\label{eqn:cJJsimplecurrent}
c_{J,J}= (-1)^Ne^{-2\pi i d}
\end{align}
where $N\in\Z$ is such that
\begin{align*}
\cY(j,x)j = vx^{-2d + N} + \sum_{n>N, n\in \Z}v_nx^{-2d+n},
\end{align*}
for any non-zero intertwining operator $\cY$ of type $V\choose J\, J$
and $v,v_n\in V$ with $v\neq 0$.  Moreover, if the category is ribbon
then,
\begin{align}
\qdim(J)=(-1)^Ne^{-4\pi i d}
\label{eqn:qdimsimplecurrent}.
\end{align}
\end{prop}
\begin{proof}
First, note that all intertwining operators of type $V \choose J\,J$
are scalar multiples of each other as $J$ is a simple current. 
Fix an isomorphism $i:J\boxtimes J \rightarrow V$.
Let $\cY$ be the intertwining operator corresponding to the 
intertwining map $\bar{i}\circ\boxtimes$.
By the definition of braiding, we know that
\begin{align*}
c_{J,J}\cY(j,1)j &= e^{L(-1)}\cY(j,e^{\pi i})j
=e^{L(-1)}\Bigl(ve^{(-2d + N)\pi i} + \sum_{n>N, n\in \Z}v_ne^{(-2d+n)\pi i}\Bigr)
= ve^{(-2d + N)\pi i} + \cdots,
\end{align*}
where the ellipses denote a sum over elements that have strictly higher weight
than $v$.
However, 
\begin{align*}
c_{J,J}\cY(j,1)j &= c_{J,J}(v + \cdots).
\end{align*}
Comparing, we arrive at equation \eqref{eqn:cJJsimplecurrent}.
If the category is ribbon, we can use the spin statistics theorem, i.e.,
Corollary \ref{cor:spinstatistics} to deduce \eqref{eqn:qdimsimplecurrent}.
\end{proof}

\subsection{Rational building blocks}
\allowdisplaybreaks
Quantum dimensions are only known for rational VOAs via the Verlinde
formula, however note that \cite{CG} suggests that this generalizes to
the $C_2$-cofinite setting.  In a unitary VOA, the quantum dimension
of a simple current is always one. In a non-unitary VOA this
is not guaranteed anymore.  Our building blocks are
\begin{enumerate}
\item The simple rational Virasoro vertex algebra $\vir(u, v)$
at central charge
$c_{u, v}=1-6(u-v)^2/(uv)$
where $u, v$ are coprime positive integers larger than two.  This VOA
is non-unitary, except if $|u-v|=1$. It has a self-dual simple current
$J_{u, v}$ of conformal dimension $h_{u, v}= (u-2)(v-2)/4$  and quantum dimension
 $\qdim\left(J_{u, v}\right) = (-1)^{u+v+1},$
which can be obtained from \cite{IK}.
\item Let $L$ be an even positive definite lattice, then the
associated lattice VOA $V_L$ has the property that it has the least
conformal dimension among all its irreducible modules, therefore, by
Proposition 4.17 and Example 4.19 of \cite{DJX}, the quantum
dimension of a simple current is one. The conformal dimension of a
simple current is given by the norm squared over two of the
corresponding coset representative. Actually, in this case,
each irreducible module is a simple current.
\item The simple affine VOA $L_k(\gs\gl_2)$ for a positive integer $k$ is
rational and unitary, it has a self-dual simple current $K_k$ of
conformal dimension $\frac{k}{4}$.
\end{enumerate}

\subsection{Logarithmic extensions}

Recall the triplet VOAs $\mathcal W(p)$ are $C_2$-cofinite
non-rational VOAs \cite{FGST, AdM, TW}. They are defined as
\[
\cW(p) = \text{ker}_Q \left(V_{\sqrt{2p}\mathbb Z}\right)
\]
where $Q$ is a certain specific screening operator that intertwines
lattice VOA modules.  $\cW(p)$ has an order two simple current,
denoted by $X_1^-$, obtained from the only self-dual simple current $J$ of $V_{\sqrt{2p}}$  in the straight-forward manner
$X_1^- = \text{ker}_Q \left(J\right)$.
The ordinary Virasoro element is actually not in the kernel, but only
a shifted version of central charge $c=1 -6(p-1)^2/p$
and the conformal dimension of the two lowest-weight states of $X^-_1$
under the Virasoro-zero mode is $(3p-2)/4$.
Now we derive that 
\begin{align}\label{eqn:qdimWp}
\qdim(X_1^-)=-(-1)^p.
\end{align}
Indeed, using the lattice realization of the $\cW(p)$ algebra and 
its simple current $X_1^-$ from \cite{AdM}, we know that the $N$
in Proposition \ref{prop:qdimsimplecurrent} is such that
$-2(3p-2)/4+N = p/2$. This immediately yields
the equation \ref{eqn:qdimWp}.
We then get two types of new logarithmic VOAs as
\[
\gA_p= \cW(p)\otimes  L_{p-2}(\gs\gl_2) \oplus X^-_1\otimes K_{p-2}
\qquad
\text{and}
\qquad
\gB_p=  
\cW(p)\otimes   \vir(3, p) \oplus X^-_1\otimes  J_{p, 3}
\]
the simple currents in the first one have now conformal dimension
$p-1$ and hence we get a VOA of correct statistics if $p$ is odd
and a super VOA of wrong statistics if $p$ is even.  In the second
case, the simple currents also have dimension $p-1$ so that for each
$p$ it is a super VOA of wrong statistics.
We also consider 
\[
\gC_p = \cW(p)\otimes\cW(p)\oplus X_1^-\otimes X_1^-.
\]
Here the conformal dimension of the simple current is 
$(3p-2)/2$, and hence, we get a VOA of correct statistics if $p$ is
even and a super VOA of correct statistics if $p$ is odd.

We would now like to employ Corollary
\ref{cor:liftingsimple} and Theorem \ref{cor:liftingprojective}
to decide how the modules lift to extensions.
We have the following list of inequivalent indecomposable modules for
$\cW(p)$, $L_k(\gs\gl_2)$ and $\vir(3,p)$ and their conformal
dimensions:

\renewcommand{\arraystretch}{2}
\begin{table}[ht]
\begin{tabular}{|c|c|c|c|}
\hline
VOA & Module & Type & Conformal weight \\
\hline
$\cW(p)$ & $X_s^\pm$, $s\in N_p$ & Simple &  $h_s^+  = \dfrac{(p-s)^2-(p-1)^2}{4p}$\\
 & $P_s^\pm$, $s\in N_{p-1}$ & Reducible, indecomposable &  $h_s^- = \dfrac{(2p-s)^2-(p-1)^2}{4p}$\\
\hline
$L_k(\gs\l_2)$ & $L((k-t)\Lambda_0+t\Lambda_1)$, $t\in N^0_{k}$ & Simple & 
$h_{(k-t)\Lambda_0+t\Lambda_1}= \dfrac{t(t+2)}{4(k+2)}$\\
\hline
$\vir(3,p)$ & $\phi_{1,s}$, $s\in N_{p-1}$ 
& Simple & $h_{\phi_{1,s}}= \dfrac{(p - 3s)^2 - (p-3)^2}{12p}$\\
\hline
\end{tabular}
\label{tab:data}\caption{Modules and conformal weights, $X^\pm_s$ and $P^\pm_s$ have conformal weight $h^\pm_s$.}
\end{table}
\renewcommand{\arraystretch}{1}
Where $N_p=\{1,2,\dots,p\}$ and $N^0_p=\{0,1,\dots,p\}$.
It is also known how the simple currents permute the indecomposables.
\begin{align*}
\cW(p)&: \quad  
X_1^-\boxtimes X_s^{\pm}  \cong X_s^\mp, \qquad
X_1^-\boxtimes P_s^{\pm}  \cong P_s^\mp,\\
L_k(\gs\gl_2)&: \quad L(k\Lambda_1) \boxtimes L((k-t)\Lambda_0 + t\Lambda_1) \cong 
L(t\Lambda_0 + (k-t)\Lambda_1),\\
\vir(3,p)&: \quad \phi_{1,p-1}\boxtimes \phi_{1,s} \cong \phi_{1,p-s}.
\end{align*}
We first look at $\gA_p$. Consider the following
classes of modules.
\begin{align*}
I_{\gA_p} &= \{   X_s^+\otimes L((p-2-t)\Lambda_0+t\Lambda_1), 
X_s^-\otimes L(t\Lambda_0+(p-2-t)\Lambda_1) \, | \, 
\\
& \qquad\qquad s\in N_p, t\in N^0_{p-2},\, s-t\not\equiv p\pmod{2}
\}\\
P_{\gA_p} &= \{   P_s^+\otimes L((p-2-t)\Lambda_0+t\Lambda_1), 
P_s^-\otimes L(t\Lambda_0+(p-2-t)\Lambda_1) \, | \, 
\\
&\qquad\qquad s\in N_{p-1},\, t\in N^0_{p-2},\, s-t\not\equiv p\pmod{2}
\}
\end{align*}
For $\gB_p$, consider the following
classes of modules.
\begin{align*}
I_{\gB_p} &= \{   X_s^+\otimes \phi_{1,t}, 
X_s^-\otimes \phi_{1,p-t} \, | \, 
s\in N_p, t\in N_{p-1}, s-t\equiv p\pmod{2}
\}\\
P_{\gB_p} &= \{   P_s^+\otimes \phi_{1,t}, 
P_s^-\otimes \phi_{1,p-t} \, | \, 
s\in N_{p-1},\, t\in N_{p-1},\, s-t\equiv p\pmod{2}
\}
\end{align*}
For $\gC_p$, consider the following
classes of modules.
\begin{align*}
I_{\gC_p} &= \{   X_s^+\otimes X_t^+,
X_s^-\otimes X_t^- | \, 
s\in N_p, t\in N_p, s+t\equiv p \pmod{2}
\}\\
& \cup
\{   X_s^+\otimes X_t^-,
X_s^-\otimes X_t^+ | \, 
s\in N_p, t\in N_p, s-t\equiv 0\pmod{2}
\}\\
P_{\gC_p} &= \{   P_s^+\otimes X_t^+,P_s^+\otimes X_t^+, P_s^+\otimes P_t^+,
X_s^-\otimes P_t^-, P_s^-\otimes X_t^-,  P_s^-\otimes P_t^-\,| \, \\
&\quad 
s\in N_p,t\in N_p,\, s+t\equiv p\pmod{2}
\}\\
& \cup
\{P_s^+\otimes X_t^-,P_s^+\otimes X_t^-, P_s^+\otimes P_t^-,
X_s^-\otimes P_t^+, P_s^-\otimes X_t^+,  P_s^-\otimes P_t^+\,| \, \\
&\quad 
s\in N_p,t\in N_p,\, s-t\equiv 0\pmod{2}
\}
\end{align*}
It is clear that the indecomposable modules have finite dimensional
endomorphism rings, because they have finite length. It is also 
known that the Jordan blocks for $L(0)$ are bounded in size.
Therefore, in each case, $I_\bullet$ lift to simple modules,
$P_\bullet$ to reducible indecomposable modules.

\begin{figure}[h]
\begin{center}
\begin{tikzpicture}[thick,>=latex,
nom/.style={circle,draw=black!20,fill=black!20,inner sep=1pt}
]
\node (top1) at (5,1.5) [] {$ A\otimes L $};
\node (left1) at (3.5,0) [] {$ B\otimes L $};
\node (right1) at (6.5,0) [] {$ B\otimes L $};
\node (bot1) at (5,-1.5) [] {$  A \otimes L$};
\draw [->] (top1) -- (left1);
\draw [->] (top1) -- (right1);
\draw [->] (left1) -- (bot1);
\draw [->] (right1) -- (bot1);

\node (top2) at (12,1.5) [] {$ \cF(A\otimes L) $};
\node (left2) at (10.5,0) [] {$ \cF(B\otimes L) $};
\node (right2) at (13.5,0) [] {$ \cF(B\otimes L) $};
\node (bot2) at (12,-1.5) [] {$  \cF(A\otimes L)$};
\draw [->] (top2) -- (left2);
\draw [->] (top2) -- (right2);
\draw [->] (left2) -- (bot2);
\draw [->] (right2) -- (bot2);
\draw [->,dotted] (right1) -- node[above] {$\cF$} (left2) ;
\end{tikzpicture}
\caption{\label{fig:Ap} An example in the the case $\gA_p$.
Here, $A=X_s^\pm$, $B =X_s^\mp$,  $L=L((p-2-t)\Lambda_0+t\Lambda_1)$.}
\end{center}
\end{figure}
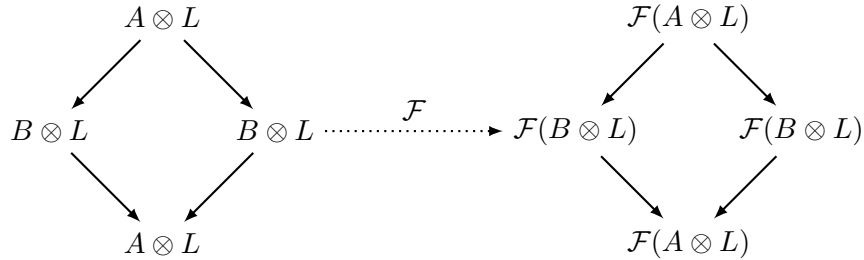

\begin{figure}[h]
\begin{center}
\begin{tikzpicture}[thick,>=latex]
\node (t11) at (11,4) [] {$ B\otimes D $};
\node (t33) at (4,4) [] {$ A \otimes C $};
\node (b11) at (11,0) [] {$B\otimes D $};
\node (b33) at (4,0) [] {$A\otimes C $};
\node (01') at (13,2) [] {$A\otimes D $};
\node (31') at (9,2) [] {$A\otimes D $};
\node (13') at (6,2) [] {$B\otimes C $};
\node (63') at (2,2) [] {$B\otimes C $};
\draw [->] (t11) -- (01');
\draw [->] (01') -- (b11);
\draw [->] (t11) -- (31');
\draw [->] (31') -- (b11);
\draw [->] (t33) -- (13');
\draw [->] (13') -- (b33);
\draw [->] (t33) -- (63');
\draw [->] (63') -- (b33);
\draw [->,dotted] (t11) -- (63');
\draw [->,dotted] (13') -- (b11);
\draw [->,dotted] (t33) -- (31');
\draw [->,dotted] (31') -- (b33);
\draw [->,dotted] (t11) -- (13');
\draw [->,dotted] (t33) -- (01');
\draw [->,dotted] (01') -- (b33);
\draw [->,dotted] (63') -- (b11);
\end{tikzpicture}
\caption{\label{fig:Cp}
An indecomposable $\cW(p)\otimes \cW(p)$-module that is 
not a tensor product of two indecomposable $\cW(p)$-modules.
Here, $A=X_s^\pm$,  $B=X_s^\mp$, $C=X_t^{\pm}$,
$D=X_t^\mp$.
The solid arrows denote the action of the left tensorand
and the dashed arrows denote the action of the right tensorand.}
\end{center}
\end{figure}
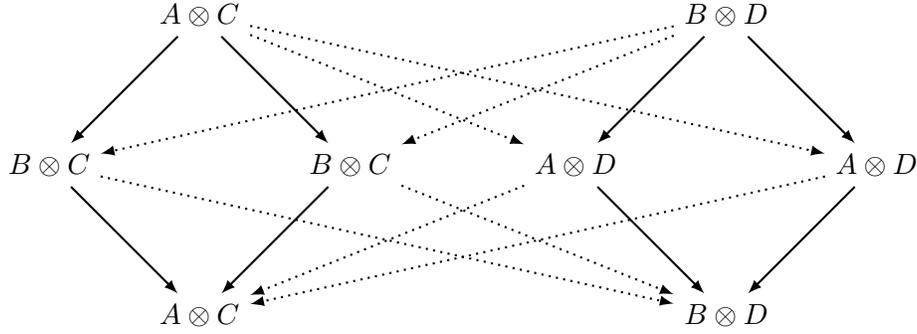

Using Remark \ref{rem:Fisexact} and using the known Loewy diagrams of
the indecomposables, one can quickly deduce Loewy diagrams of the induced modules.
In Figure \ref{fig:Ap} we show an example in the case $\gA_p$.
For $\cW(p)\otimes \cW(p)$ one can have indecomposable modules that 
are not tensor products of indecomposable modules for the individual tensor
factors, for example, see Figure \ref{fig:Cp}, (cf.\ \cite{CR3}).
These modules again induce as in Figure \ref{fig:Ap}.

\subsection{On W-algebras}

The following conjecture is from the physics literature \cite{BEHHH}.
\begin{conj}\label{conj:W}
Let $\cW^{(2)}_r$ be the Feigin-Semikhatov algebra \cite{FS} of level
$k=\frac{n-r^2+2r}{r-1}$ then
\[
\com\left(H, \cW^{(2)}_r\right) \cong W_{A_{n-1}}\left(n+1, n+r\right).
\]
Especially $\cW^{(2)}_r$ is rational.
\end{conj}
\begin{remark}The Feigin-Semikhatov algebra $\cW^{(2)}_r$ in turn is
believed to be a quantum Hamiltonian reduction of $V_k(\gs\gl_r)$
for a certain non-principal nilpotent element. This conjecture is true for
$r=2, 3$ as only in these two cases all involved OPEs are known.
\end{remark}
Conjecture \ref{conj:W} is true for $r=3$ \cite{ACL}. The proof uses
our results.  The following is immediate from the previous subsection:
\begin{prop}
Let $L=\sqrt{2r}\mathbb Z$ and let $J$ be its
only self-dual simple current, then
\[
\vir(3, 2+r)\otimes V_L\oplus J_{3, 2+r} \otimes J
\]
is a W-algebra that is strongly generated by two dimension
$\frac{r}{2}$ fields, a Heisenberg field and the Virasoro field and
has same central charge as $\cW^{(2)}_r$ at level
$k=\frac{2-r^2+2r}{r-1}$.
\end{prop}
\begin{proof}
The conformal dimension of $J_{3, 2+r} \otimes J$ is by construction
the desired $\frac{r}{2}$ and the quantum dimension is $(-1)^r$ so
that the resulting extension is always a VOA.  It is strongly
generated by the strong generators of $\vir(3, 2+r)\otimes V_L$
together with the fields of the two lowest-weight vectors of
$J_{3, 2+r} \otimes J$. However, by Proposition \ref{prop:J-OPE} together with the well-known lattice VOA operator products
the two strong generators of $V_L$ of conformal dimension $r$ must
be normal ordered products of the two lowest-weight vectors of
$J_{3, 2+r} \otimes J$ with themselves.
\end{proof}
\begin{remark}
Call the dimension $\frac{r}{2}$ fields $G^\pm$ and the dimension
one field $J$, then using the well-known operator products of
lattice VOAs it is easy to verify that with appropriate
normalization of $G^\pm$ and $J$, the OPE of $J$ with itself as well
as the one of $J$ with $G^\pm$ coincides with the OPE of the
corresponding fields of $W^{(2)}_r$ at level
$k=\frac{2-r^2+2r}{r-1}$ as given in \cite{FS}. In the case of the
OPE of $G^+$ with $G^-$ only the first two leading OPE coefficients
can be easily computed and they again coincide with those given in
\cite{FS}.
\end{remark}
\begin{remark}
Tweaking the lattice a bit, one gets a W super algebra, that is
believed to be an affine W-super algebra associated to
$\gs\gl(r|1)$. This belief is basically due to \cite{FS}.  Namely,
let
$$N= \sqrt{2(r+2)}\mathbb Z$$ and let $K$ be the unique
self-dual simple current of $V_N$, then the lowest-weight states of
$J_{3, 2+r} \otimes K$ have conformal dimension $\frac{r+1}{2}$ so
that
\[
\vir(3, 2+r)\otimes V_N\oplus J_{n+1, n+r} \otimes K
\]
is now a super VOA. In the case $r=2$ this is the $N=2$ super
conformal algebra, which has already known to be rational \cite{Ad3}.
\end{remark}

\subsection{More super VOAs}

We believe the following:
\begin{conj}\label{conj:osp}
$L_k(\gs\gl_2)$ is a subVOA of $L_k(\go\gs\gp(1|2))$ and
\begin{align}
\com\left(L_k(\gs\gl_2), L_k(\go\gs\gp(1|2)\right))\cong \vir(k+2, 2k+3)
\end{align}
for
positive integer $k$, especially $L_k(\go\gs\gp(1|2))$ is a
rational super VOA.
\end{conj}
The conjecture is motivated from \cite{CL}.  Namely, it was shown that
the universal coset VOA
$\com\left(V_k(\gs\gl_2), V_k(\go\gs\gp(1|2)\right))$ is just the
universal Virasoro algebra for generic $k$.  Also computational
evidence was given that integral $k$ are generic, and further it was
shown that the coset of the universal VOAs surjects on the coset of
the corresponding simple quotients. In other words, according to
\cite{CL} the conjecture is true if $L_k(\gs\gl_2)$ is a subVOA of
$L_k(\go\gs\gp(1|2))$ and if positive integer $k$ is generic.

Using the singlet and triplet algebras, one can construct interesting
new logarithmic VOAs, examples are algebras of Feigin-Semikhatov type
\cite{FS} constructed in \cite{Ad1, CRW}, but also the small $N=4$
super conformal algebra at central charge $-9$ \cite{Ad2}. Here, we
will give two further examples and in the same manner prove above
conjecture for $k=1$.

\begin{thm}\label{thm:osp} 
Conjecture \ref{conj:osp} is true for $k=1$. Also for $k=-1/2$ the commutant is
$\vir(3, 4)$.
\end{thm}
\begin{proof}

(1) We first look at $k=1$.
We need to prove that:
$$ V = (L_1(\mathfrak{sl}_2)\otimes \vir(3,5)) \oplus 
(J_{\mathfrak{sl}_2}\otimes J_{\vir}) \cong L_1(\go\gs\gp(1|2)).$$
First note that $L_1(\mathfrak{sl}_2)\otimes \vir(3,5)$ is a
simple VOA, being a tensor product of simple VOAs. Simple current
extensions of simple VOAs are simple and hence, $V$ is simple.

It is well known that $\vir(3,5)$ is a finitely reductive vertex
operator algebra such that $\vir(3,5)\cong \vir(3,5)'$ and such that
$\vir(3,5)_n=0$ for all $n<0$. Therefore, we can invoke equation
\eqref{eq:non-degeneracy-useful}.  $L_1(\gs\gl_2)\cong V_{\sqrt{2}\Z}$
is a lattice VOA. Denote by $\phi_\lambda$ the vertex operator
associated to the lattice vector $\lambda$. Then the three currents
are $e=\phi_{\sqrt{2}}, f=\phi_{-\sqrt{2}}$ and the Heisenberg field
$h$, the self-dual simple currents has two fields
$x=\phi_{1/\sqrt{2}}, y=\phi_{-1/\sqrt{2}}$ of conformal dimension
$1/4$.  The simple current field (associated to the lowest weight
vector of $J_\text{Vir}$) of $\vir(3, 5)$ we denote by $J$.  Now,
consider the five dimension 1 fields:
$(e\otimes {\bf 1})(z),(f\otimes{\bf 1})(z),(h\otimes{\bf
  1})(z),(x\otimes J)(z),(y\otimes J)(z)$.  We know that
\begin{align*}
J(z)J(w) &\sim  (z-w)^{-3/2}(\ell + \dots), \qquad\qquad 
x(z)y(w) \sim (z-w)^{-1/2}(1 + (z-w)h(w)+ \dots),\\
x(z)x(w) &\sim (z-w)^{1/2}(e(w) + \dots),\qquad\qquad
y(z)y(w) \sim (z-w)^{1/2}(f(w) + \dots)
\end{align*}
and $\ell$ is non-zero by equation \eqref{eq:non-degeneracy-useful}.
Now, one can easily verify OPEs to prove that these five
fields generate a vertex subalgebra isomorphic to a quotient of
$V_1(\go\gs\gp(1|2))$.

Finally 
$\omega_{\go\gs\gp}-(\omega_{\gs\gl_2}\otimes{\bf 1}
)\in\com(L_1(\mathfrak{sl}_2),V)$ and ${\bf 1}\otimes\omega_{\vir}(3,5)\in\com(L_1(\mathfrak{sl}_2),V)$
are both conformal vectors and since
$\com(L_1(\gs\gl_2),V)=\vir(3,5)$
they coincide.
Therefore, we see that
these five dimension 1 fields strongly generate the entire $V$.
Hence, the entire $V$ is a quotient of $V_1(\go\gs\gp(1|2))$. Since
$V$ is simple, we get that $V\cong L_1(\go\gs\gp(1|2))$.

(2) Now we look at $k=-1/2$.
As before, let
$$ V = (L_{-1/2}(\mathfrak{sl}_2)\otimes \vir(3,4)) \oplus 
(J_{\gs\gl_2}\otimes J_{\vir}).$$
We know that
$L_{-1/2}(\mathfrak{sl}_2)\oplus J_{\mathfrak{sl}_2} \cong \cS(1)$ the
rank one $\beta\gamma$-VOA and hence
$L_{-1/2}(\mathfrak{sl}_2)\cong (\cS(1))^{\Z/2\Z}$. Also
$\vir(3,4)\oplus J_{\vir} \cong \cF(1)$, the free fermion super VOA
and thus $\vir(3,4)\cong (\cF(1))^{\Z/2\Z}$.  Therefore, the group
$\Z/2\Z\times\Z/2\Z$ acts by automorphisms on $\cS(1)\otimes \cF(1)$
and
$$V = (\cS(1)\otimes \cF(1))^G$$ where 
$G=\{(0,0),(1,1)\}\subset \Z/2\Z\times \Z/2\Z$.
Hence, $V$ is a super VOA.
By the same reasoning as in the previous case, $V$ is simple.  Just
like before, one can check the OPEs of the dimension 1 fields to prove
that these five fields generate a vertex subalgebra isomorphic to a
quotient of $V_{-1/2}(\go\gs\gp(1|2))$, which must be the entire $V$
by analogous arguments.
\end{proof}

\begin{remark}
In the $k=-1/2$ case, we haven't really used Theorem
\ref{thm:constructingVOSA}, instead we have given an indirect proof
that the ``simple current'' extension is a super VOA.  If one were to use
Theorem \ref{thm:constructingVOSA}, one first has to prove that
Huang-Lepowsky-Zhang's theory can be invoked in this case.
We expect the results of \cite{Miy} to be useful.
For Theorem \ref{thm:N=4} also we shall give an indirect proof.
\end{remark}

We would now like to analyze the irreducible modules of $L_1(\go\gs\gp(1|2))$.
\begin{lemma}\label{lem:inducedsimples}
Let $V_e=V\oplus J$ be a simple current extension by a 
simple current $J$ of finite order. Assume also
that $J\boxtimes W \not\cong W$ for a simple $V$-module $W$.
Let $W_e$ be a simple $V_e$-module that
contains a simple $V$-submodule $W$.  Then,
$W_e \cong \bigoplus_{i\in G} J^i\boxtimes W$, where
$G$ is the finite cyclic group generated by $J$.
\end{lemma}
\begin{proof}
The proof is the same as the proof of Theorem 3.7 in \cite{La},
except that we do not require $V$ to be rational.
For $S\subset V_e$ let
\begin{align}
S\cdot W = \text{Span}\{j_nw\,|\,j\in S, w\in W, n\in\Z \}.
\end{align}
It is clear that $J^i\cdot W$ is a $V$-submodule of $W_e$ for each
$i\in G$.  Moreover, restriction of the $V_e$-module map provides an
intertwining operator, say $\cY$, of type
$J^i\cdot W \choose J^i\, W$.  Now, if for some $j\in J^i,w\in W$,
$j_nw=0$ for all $n$, then $J^i\cdot W =0$ by Proposition 11.9 of
\cite{DL}. It is easy to see that this implies $V_e\cdot W=0$.
Therefore, we conclude that $J^i\cdot W\neq 0$, and hence $\cY\neq 0$.
By universal property of tensor products, there exists a non-zero
$V$-module map $J^i\boxtimes W \rightarrow J^i\cdot W$.  This map is
clearly surjective. Since $J^i\boxtimes W$ is simple, this map is an
isomorphism. We identify $J^i\cdot W$ with $J^i\boxtimes W$. We have
that $\sum_{i\in G}J^i\boxtimes W$ is a submodule of $W_e$ and hence
$W_e=\sum_{i\in G}J^i\boxtimes W$. The sum is direct because of the
assumption that $J\boxtimes W\not\cong W$ for any simple $V$-module
$W$.
\end{proof}

\begin{remark}
The condition on the module $W_e$ of Lemma \ref{lem:inducedsimples}
is satisfied if $W_e$ has finite length as a $V$-module. 
See \cite{H5} for conditions under which this is guaranteed to happen.
\end{remark}

\begin{cor}\label{cor:ospmodules}
$L_1(\go\gs\gp(1|2))$ has precisely two inequivalent simple modules.
\end{cor}
\begin{proof}

Let $V=L_1(\gs\gl_2) \otimes \vir(3, 5)$,
$J = L(\Lambda_1)\otimes J_{\vir}$ and let
$V_e=L_1(\go\gs\gp(1|2))\cong V\oplus J$.  From \cite{Li}, 
$V$ is rational. Now, we can invoke Lemma \ref{lem:inducedsimples}
to gather that any simple $V_e$-module is of the form
$W\oplus J\boxtimes W$ for some simple $V$-module $W$.
By Proposition 4.7.4 of \cite{FHL}, any simple module for
$L_1(\gs\gl_2) \otimes \vir(3, 5)$ 
is a tensor product of simple modules for $L_1(\gs\gl_2)$ 
and $\vir(3, 5)$.
Gathering the data from Table \ref{tab:data}, it is clear that
\[
M_1 = L(\Lambda_0)\otimes \phi_{1, 1},
M_2 = L(\Lambda_0)\otimes \phi_{1, 3}, 
M_3 = L(\Lambda_1)\otimes \phi_{1, 2},
M_4 = L(\Lambda_1)\otimes \phi_{1, 4}
\]
are the only $L_1(\gs\gl_2) \otimes \vir(3, 5)$-modules that lift to
$L_1(\go\gs\gp(1|2))$-modules by Corollary \ref{cor:liftingsimple}.
However,  $M_1$ and $M_4$ lift to isomorphic modules and so do $M_2$
and $M_3$.
\end{proof}

\begin{thm}\label{thm:N=4}
Let $\cW(2)$ be the $C_2$-cofinite $c=-2$ triplet algebra, then
$\cW(2)\otimes L_{-1/2}(sl_2)$ has a simple current extension
isomorphic to the small $N=4$ super Virasoro algebra at $c=-3$.
\end{thm}
\begin{proof}
Let
$$V=X_1^+\otimes L_{-1/2}(\mathfrak{sl}_2) \oplus X_1^{-}\otimes
J_{\mathfrak{sl}_2}.$$
We have as before that
$L_{-1/2}(\mathfrak{sl}_2)\oplus J_{\mathfrak{sl}_2}\cong \cS(1)$ (the
$\beta\gamma$-VOA) and thus
$L_{-1/2}(\mathfrak{sl}_2)\cong (\cS(1))^{\Z/2\Z}$, but also
$X_1^+\oplus X_1^- \cong \cA(1)$, the rank one symplectic fermion
super VOA so that $X_1^+\cong (\cA(1))^{\Z/2\Z}$. Therefore, the group
$\Z/2\Z\times\Z/2\Z$ acts by automorphisms on $\cA(1)\otimes \cS(1)$
and $V = (\cA(1)\otimes \cS(1))^G$ where
$G=\{(0,0),(1,1)\}\subset \Z/2\Z\times \Z/2\Z$.  Hence, $V$ is a
vertex operator algebra.

By reasoning as before, $V$ is simple.

The small $N=4$ super Virasoro algebra is generated by three
$\mathfrak{sl}_2$ fields of weight 1, one Virasoro field of weight 2
and its four superpartners of weight $3/2$.  We have got the required
number of fields.

We must check that the OPEs match the ones for the small $N=4$.  We
denote the lowest weight states of $J_{\mathfrak{sl}_2}$ by
$\beta,\gamma$.  We denote the lowest weight states of $X_1^-$ by
$s^+,s^-$.

We know that:
\begin{align*}
\beta(z)\beta(w) &\sim 2e(w) + (z-w)(\beta(-3/2)\beta(-1/2){\bf 1})(w)+\dots\\
\gamma(z)\gamma(w) &\sim 2f(w)  + (z-w)(\gamma(-3/2)\gamma(-1/2){\bf 1})(w)+\dots\\
\beta(z)\gamma(w) &\sim -(z-w)^{-1} + h(w) + (z-w)(\beta(-3/2)\gamma(-1/2){\bf 1})(w) + (z-w)^2(\beta(-5/2)\gamma(-1/2){\bf 1})(w)+   \dots\\
\gamma(z)\beta(w) &\sim (z-w)^{-1} + h(w) + (z-w)(\gamma(-3/2)\beta(-1/2){\bf 1})(w)+  (z-w)^2(\gamma(-5/2)\beta(-1/2){\bf 1})(w)+\dots\\
s^+(z)s^-(w) &\sim (z-w)^{-2} + (s^+(-1)s^-(-1){\bf 1})(w) + (z-w)(s^+(-2)s^-(-1){\bf 1})(w)+\dots\\
s^-(z)s^+(w) &\sim -(z-w)^{-2}+ (s^-(-1)s^+(-1){\bf 1})(w) + (z-w)(s^-(-2)s^+(-1){\bf 1})(w)\dots\\
s^+(z)s^+(w)&\sim (z-w)(\dots)\\
s^-(z)s^-(w)&\sim (z-w)(\dots).
\end{align*}
Let $J^+=-\frac{1}{2}:\beta\beta:$, $J^-=\frac{1}{2}:\gamma\gamma$ and
$h=:\beta\gamma:$, then the OPE of these three is the operator product
algebra of $L_{-1/2}(\gs\gl_2)$ \cite{R}.  Let
\[
G^+= \beta\otimes s^+,\qquad  G^- =\gamma\otimes s^+, \qquad
\bar{G}^+= -\beta\otimes s^-, \qquad\bar{G}^- =\gamma\otimes s^-. 
\]
From the OPEs above, it is easy to calculate the $\lambda$-brackets as
in \cite{KRW}:

Since the OPE of $v_1,v_2$ for any $v_1,v_2\in\{G^+,G^-\}$ or
$v_1,v_2\in\{\bar{G}^+,\bar{G}^-\}$ is regular, their
$\lambda$-bracket is 0.  We also have the following OPEs:
\begin{align*}
G^+(z)\bar{G}^+(w) &\sim  (z-w)^{-2}\cdot 2e(w) + (z-w)^{-1}(\beta(-3/2)\beta(-1/2){\bf 1})(w)+\dots\\
G^-(z)\bar{G}^-(w) &\sim  (z-w)^{-2}\cdot 2f(w) + (z-w)^{-1}(\gamma(-3/2)\gamma(-1/2){\bf 1})(W)+\dots\\
G^+(z)\bar{G}^-(w) &\sim  -(z-w)^{-3}+ (z-w)^{-2}\cdot h(w) + (z-w)^{-1}(\beta(-3/2)\gamma(-1/2){\bf 1} - s^+(-1)s^-(-1){\bf 1})(w)+\dots\\
G^-(z)\bar{G}^+(w) &\sim  -(z-w)^{-3}- (z-w)^{-2}\cdot h(w) -(z-w)^{-1}(\gamma(-3/2)\beta(-1/2){\bf 1} + s^+(-1)s^-(-1){\bf 1})(w)+\dots
\end{align*}
We know that $\omega_{\cW (2)}= s^-(-1)s^+(-1){\bf 1}$ and
$\omega_{\mathfrak{sl}_2}=\dfrac{1}{2}[\beta(-3/2)\gamma(-1/2) -
\gamma(-3/2)\beta(-1/2)]{\bf 1}$.
So, the $\lambda$-brackets  come out to be:
\begin{equation}\nonumber
\begin{split}
[J^\pm{}_\lambda G^\mp] &= G^\pm,\qquad
[J^\pm{}_\lambda \bar{G}^\mp] = -\bar{G}^\pm,\qquad
[G^\pm{}_\lambda \bar{G}^\pm] = (\partial + 2\lambda)J^\pm,\\
[G^+{}_\lambda \bar{G}^-]&=-\dfrac{1}{2}\lambda^2 + \lambda J^0 + L + \dfrac{1}{2}\partial J^0,\qquad
[G^-{}_\lambda \bar{G}^+]=-\dfrac{1}{2}\lambda^2 - \lambda J^0 + L -  \frac{1}{2}\partial J^0.
\end{split}
\end{equation}
Therefore, we've got the correct $\lambda$-bracket structure for the
small $N=4$ super Virasoro algebra at $c=-3$.

Finally we verify that these fields strongly generate $V$. For this we have to check that we can obtain the element
$s^-(-2)s^+ + s^+(-2)s^-\in \cW(2)$ as a normally ordered polynomial in the other generators and their derivatives.  We know the following (the
subscript denotes the mode):
\begin{align*}
(\beta s^+)_{-1}(\gamma s^-) &= -s^+(-2)s^- + hs^+s^-+\beta(-5/2)\gamma \\
(\gamma s^+)_{-1}(\beta s^-) &= s^+(-2)s^- + hs^+s^-+\gamma(-5/2)\beta \\
(\beta s^-)_{-1}(\gamma s^+) &= -s^-(-2)s^+ + hs^-s^+-\beta(-5/2)\gamma \\
(\gamma s^-)_{-1}(\beta s^+) &= s^-(-2)s^+ + hs^-s^+-\gamma(-5/2)\beta.
\end{align*}
Therefore,
$$ 
(\gamma s^+)_{-1}(\beta s^-) + (\gamma s^-)_{-1}(\beta s^+) 
-(\beta s^+)_{-1}(\gamma s^-) - (\beta s^-)_{-1}(\gamma s^+) 
= s^+(-2)s^- + s^-(-2)s^+.
$$
\end{proof}

\subsection{Some orbifolds with categories of $C_1$-cofinite modules}

Our results apply to module categories of VOAs that are vertex tensor categories in the sense of Huang-Lepowsky. The main obstacle for the conditions of Theorem \ref{thm:cond2} is in verifying the $C_1$-cofiniteness of modules. We thus close this work with a few examples on this question in the context of orbifolds of free field algebras. 

Consider first the rank $n$ Heisenberg vertex algebra $\cH(n)$, whose full automorphism group is the orthogonal group $\text{O}(n)$. By \cite{DLMI}, there is a dual reductive pair decomposition $$
\cH(n) = \bigoplus_{\nu} L_{\nu} \otimes M^{\nu},$$ where the sum is over all finite-dimensional, irreducible $\text{O}(n)$-modules $L_{\nu}$, and the $M^{\nu}$'s are inequivalent, irreducible $\cH(n)^{\text{O}(n)}$-modules. 

The $C_1$-cofiniteness of the $\cH(n)^{\text{O}(n)}$-modules $M^{\nu}$ was established in \cite{LIII,LIV}, and we briefly sketch the proof. First, we may view $\cH(n)^{\text{O}(n)}$ as a deformation of the classical invariant ring $R = (\text{Sym} \bigoplus_{k\geq 0} V_k)^{\text{O}(n)}$, where $V_k \cong \mathbb{C}^n$ as $\text{O}(n)$-modules. In particular, $\cH(n)$ admits an $\text{O}(n)$-invariant good increasing filtration in the sense of Li \cite{LiII}, and $\text{gr}(\cH(n)^{\text{O}(n)}) \cong R$ as differential graded commutative rings. Using Weyl's first fundamental theorem of invariant theory for $\text{O}(n)$ \cite{We}, it is not difficult to find an (infinite) strong generating set for $\text{gr}(\cH(n)^{\text{O}(n)})$ consisting of an element in each weight $2,4,6,\dots$. A consequence is that the Zhu algebra of $\cH(n)^{\text{O}(n)}$ is abelian. This implies that all irreducible, admissible $\cH(n)^{\text{O}(n)}$-modules are highest-weight modules, i.e., they are generated by a single vector. In particular, this holds for each $M^{\nu}$ above.

It follows from Weyl's second fundamental theorem of invariant theory for $\text{O}(n)$ \cite{We} that the relation of minimal weight among the generators of $\cH(n)^{\text{O}(n)}$ occurs at weight $n^2+3n+2$. In \cite{LIII}, it was conjectured that this gives rise to a decoupling relation expressing the generator in weight $n^2+3n+2$ as a normally ordered polynomial in the generators of lower weight. Starting with this relation, it is easy to construct decoupling relations for all the higher weight generators, so that $\cH(n)^{\text{O}(n)}$ is of type $\cW(2,4,\dots, n^2+3n)$. In the case $n=1$, the fact that $\cH(1)^{\mathbb{Z}/2\mathbb{Z}}$ is of type $\cW(2,4)$ is a celebrated theorem of Dong and Nagatomo \cite{DN}, and this conjecture was verified for $n\leq 6$ in \cite{LIV}. Even though it remains open in general, the strong finite generation of $\cH(n)^{\text{O}(n)}$ was established for all $n$ in \cite{LIV}.

The proof that each $M^{\nu}$ is $C_1$-cofinite depends on the strong finite generation of $\cH(n)^{\text{O}(n)}$, together with the fact that the non-negative Fourier modes of the generators of $\cH(n)^{\text{O}(n)}$ preserve the filtration on $\cH(n)$. Note that Lemma 6.7 of \cite{LIII} is precisely the statement that each $M^{\nu}$ is $C_1$-cofinite according to Miyamoto's definition \cite{Miy}. This was originally proven modulo the above conjecture in \cite{LIII}, but the proof only requires that $\cH(n)^{\text{O}(n)}$ is strongly finitely generated. Therefore Lemma 6.7 of \cite{LIII} holds unconditionally.

Similar results have been established for several other orbifolds of free field algebras, and the proof is the same. The key ingredients are the strong finite generation of the orbifold and the fact that the non-negative Fourier modes of the generators preserve a filtration on the free field algebra. For the rank $n$ $\beta\gamma$-system $\cS(n)$, $\cS(n)^{\text{GL}(n)}$ and $\cS(n)^{\text{Sp}(2n)}$ are of types $\cW(1,2,\dots, n^2+2n)$ and $\cW(2,4,\dots, 2n^2+4n)$, respectively, and every irreducible submodule of $\cS(n)$ for either of these orbifolds is $C_1$-cofinite \cite{LI,LII,LIV}. The same holds for the rank $n$ $bc$-system $\cE(n)$ and the orbifold $\cE(n)^{\text{GL}(n)}$ which is isomorphic to $\cW(\mathfrak{g}\mathfrak{l}_n)$ with central charge $n$ \cite{FKRW}. Similarly, it holds for the free fermion algebra $\cF(n)$ and the orbifold $\cF(n)^{\text{O}(n)}$, which is of type $\cW(2,4,\dots, 2n)$ \cite{LV}. Finally, it holds for the rank $n$ symplectic fermion algebra $\cA(n)$ and the orbifolds $\cA(n)^{\text{Sp}(2n)}$ and $\cA(n)^{\text{GL}(n)}$, which are of types $\cW(2,4,\dots, 2n)$ and $\cW(2,3,\dots, 2n+1)$, respectively \cite{CLI}. In all these cases, the orbifolds have abelian Zhu algebras. This makes the arguments easier, but it is not essential and we expect the $C_1$-cofiniteness to hold for a  more general class of orbifolds of free field algebras. There are a few other examples in \cite{CL,CL3} where an explicit minimal strong generating set has been found using similar methods, including $(\cE(n) \otimes \cS(n))^{\text{GL}(n)}$, $(\cA(n) \otimes \cS(n))^{\text{Sp}(2n)}$, $(\cA(n) \otimes \cS(n))^{\text{GL}(n)}$, and $(\cH(n) \otimes \cF(n))^{\text{O}(n)}$. In these cases, all irreducible modules for the orbifold inside the ambient free field algebra can be shown to be $C_1$-cofinite using similar ideas.

\end{document}